\documentclass[a4paper, 11pt]{amsart}
\usepackage{subcaption}

\usepackage{amsmath}
\usepackage{amssymb}
\usepackage{epsfig}
\usepackage{eepic}
\usepackage{latexsym
}

\usepackage{dsfont}
\newtheorem{theorem}{Theorem}[section]
\newtheorem{proposition}[theorem]{Proposition}
\newtheorem{lemma}[theorem]{Lemma}

\newtheorem{definition}[theorem]{Definition}

\newtheorem{assumption}[theorem]{Assumption}

\newtheorem{remark}{Remark}[section]

\usepackage[linesnumbered,commentsnumbered,ruled,vlined]{algorithm2e}

\newcommand{\bbE}{\mathbb{E}}
\newcommand{\bbR}{\mathbb{R}}

\newcommand{\bu}{{\bf u}}
\newcommand{\F}{{\bf F}}
\newcommand{\bI}{{\bf I}}
\newcommand{\bS}{{\bf S}}
\newcommand{\bA}{{\bf A}}

\newcommand{\T}{\mathcal{T}}

\newcommand{\bsve}{{\boldsymbol \varepsilon}}

\def \ve{\varepsilon}
\def\re{{\rm e}}
\def\rg{{\rm g}}
\def\bbf{{\bf f}}

\def\la{{\langle}}
\def\ra{{\rangle}}

\usepackage[top=1.9 cm,bottom=1.9 cm,left=1.3 cm, right=1.3 cm]{geometry}

\def\XXint#1#2#3{{\setbox0=\hbox{$#1{#2#3}{\int}$}
\vcenter{\hbox{$#2#3$}}\kern-.87\wd0}}

\def\XXiint#1#2#3{{\setbox0=\hbox{$#1{#2#3}{\int}$}
\vcenter{\hbox{$#2#3$}}\kern-1.05\wd0}}

\def\XXintt#1#2#3{{\setbox0=\hbox{$#1{#2#3}{\int}$}
\vcenter{\hbox{$#2#3$}}\kern-.72\wd0}}

\def\Xinttt#1{\mathchoice
{\XXinttt\displaystyle\textstyle{#1}}%
{\XXinttt\textstyle\scriptstyle{#1}}%
{\XXinttt\scriptstyle\scriptscriptstyle{#1}}%
{\XXinttt\scriptscriptstyle\scriptscriptstyle{#1}}%
\!\int}
\def\XXinttt#1#2#3{{\setbox0=\hbox{$#1{#2#3}{\int}$}
\vcenter{\hbox{$#2#3$}}\kern-.52\wd0}}

\def\XXintttr#1#2#3{{\setbox0=\hbox{$#1{#2#3}{\int}$}
\vcenter{\hbox{$#2#3$}}\kern-.6\wd0}}

\def\XXintttt#1#2#3{{\setbox0=\hbox{$#1{#2#3}{\int}$}
\vcenter{\hbox{$#2#3$}}\kern-.78\wd0}}

\def\ddashinttt{\Xinttt-}

\begin{document}

\title{Multiscale modelling and analysis of growth of plant tissues}
\author[A. Boudaoud, A. Kiss,  M. Ptashnyk]{{Arezki Boudaoud, Annamaria Kiss, Mariya Ptashnyk }
 \medskip \medskip\\
   {\tiny LadHyX, CNRS, Ecole Polytechnique, IP Paris, Palaiseau, France} \\
   {\tiny RDP, ENS de Lyon,  Claude Bernard University Lyon 1, CNRS, INRAE, Lyon, France} \\
   {\tiny Department of Mathematics, Heriot-Watt University, The Maxwell Institute for Mathematical Sciences, Edinburgh, Scotland, UK}
  }
\date{}

% REQUIRED
\begin{abstract}
  How morphogenesis depends on cell properties is an active direction of research. Here, we focus on mechanical models of growing plant tissues, where microscopic (sub)cellular structure is taken into account. In order to establish links between microscopic and macroscopic tissue properties, we perform a multiscale analysis of a model of growing plant tissue with subcellular resolution. We use homogenization to rigorously derive the corresponding macroscopic tissue scale model. Tissue scale mechanical properties are computed from  microscopic structural and material properties, taking into account deformation by the growth field. We then consider case studies and numerically compare the detailed microscopic model and the tissue-scale model, both implemented using finite element method. We find that the macroscopic model can be used to efficiently make predictions about several configurations of interest. Our work will help making links between microscopic measurements and macroscopic observations in growing tissues.
\end{abstract}

\maketitle
%
% % REQUIRED
% \begin{keywords}
%   homogenization, multiscale modelling, plant tissue growth, linear elasticity, finite elements method, FreeFEM
% \end{keywords}
%
% % REQUIRED
% \begin{AMS}
%   35B27, 74S05, 74Qxx, 92C80, 65M60
% \end{AMS}
%

%%%%%%%%%%%%%%%%%%%%%%%%%%%%%%%%%%
\section{Introduction}
%%%%%%%%%%%%%%%%%%%%%%%%%%%%%%%%%%
Modelling plant growth and morphogenesis is an active area of research~\cite{Smithers_2019}. A major difficulty in this area is that growth is `inherently a multiscale process' \cite{Smithers_2019}, bridging subcellular to organ scales. Each plant cell is surrounded by a thin layer of polysaccharides, known as the cell wall, and exerts on this wall a hydrodynamic pressure, termed turgor pressure. Plant cell growth is driven by turgor pressure and restrained by the cell wall. Organ morphogenesis results from specific spatial distributions of growth rates across cells and tissues. As a consequence, the form of an organ is dependent on all processes occurring from sub-wall scale to supra-cellular scale, raising the need for multiscale modelling approaches.

Most previous modelling effort has focused on one scale or one process at a time. For instance, several studies have addressed how cell wall rheology emerges from its composition, microstructure, and/or synthesis~\cite{Bozorg_2014, Dyson_2012, Dyson_2010, Huang_20152go, Oliveri}. These studies adopted approaches from continuum mechanics and used partial differential equations (PDEs) to describe the cell wall as an anisotropic material, with different assumptions on the rheology -- viscous~\cite{Dyson_2012,Dyson_2010}, elastic~\cite{Bozorg_2014, Bozorg_2016, Oliveri, Ptashnyk_Seguin_2016}, or viscoelastoplastic~\cite{Huang_20152go}; growth then corresponds to flow of the viscous material or to remodelling of the elastic material, combined with synthesis of new material. Other studies represented tissues as tessellations by polygons of 2D space or of surfaces embedded in 3D that  describe the position of vertices, relying on high-dimensional systems of ordinary differential equations~\cite{Bozorg_2014, Corson_2009, Hamant_2008, Khadka_2019, Merks_2011}. The authors investigated how tissue shape changes according to cell wall rheology, assuming for instance each edge to be a viscoelastoplastic element, or to how mechanical stress in tissues feeds back on cell wall dynamics. Other studies used continuous, PDE-based approaches for modelling tissue dynamics and investigated how patterns of tissue mechanical properties or of growth potential yield organ morphogenesis~\cite{Hervieux,Vandiver}, see the comprehensive overview in~\cite{Goriely_2017}.

In this context, a systematic derivation of macroscopic (supra-cellular) tissue rheology from (sub)wall rheology is still lacking.  Here we try to address this issue using multiscale modelling and homogenization techniques. There are many results on homogenization of equations of linear elasticity, e.g.~\cite{Allaire_2002, CD, JKO,    oliveira_asymptotic_2009, pinho-da-cruz_asymptotic_2009, SP_1980}, however to our knowledge the multiscale analysis of the two-way coupling between elastic deformation and growth presented in here is novel.

In the derivation of the microscopic model we start from the framework of morphoelasticity~\cite{Goriely_2007, Rodriguez}. Following \cite{Rodriguez}, we consider the multiplicative decomposition of the deformation gradient into elastic and growth part and model elastic deformations of plant cell walls using equations of linearised elasticity. To describe the growth we use Lockhart's law~\cite{Lockhart} that relates plant cell wall growth  with deformation gradient and  accounts for microscopic subwall properties.
Along with modelling growth processes, the multiplicative decomposition approach is also used in multiscale modelling and analysis of plastic deformations, where the  decomposition  corresponds to the elastic and plastic deformations respectively, e.g.~\cite{DGP_2022, MS_2013, Neff_2005}.
%An overview on modelling growth of biological tissues  can be found in~\cite{Goriely_2017}. 
Rigorous well-posedness results were recently obtained for a model combining multiplicative decomposition and nonlinear elasticity~\cite{DNS_2022}.

%(the rigorous justification of the classical linearization approach in finite-strain elastoplasticity.)
%Similar idea was used in \cite{RSGRMP_2018} to model .... 
%Derivation of the growth equations for a unifrom body from the thermodynamical laws was considered in \cite{EM_2000}. 

Applying  homogenization  techniques, the two-scale convergence \cite{Allaire92, Nguetseng89} and periodic unfolding \cite{CDDGZ, CDG_book} methods,  we rigorously derive macroscopic equations that describe  elastic deformations and growth at the tissue scale. In particular, we provide explicit formulae for macroscopic elastic properties as a function of microscopic parameters and microscopic structure. An important step in the rigorous derivation of the macroscopic problem is the proof of the strong convergence of the sequence of growth and strain tensors as the small parameter, representing a ratio between the size of the microstructure and the size of the tissue, converges to zero. 
We illustrate our results by solving numerically the corresponding microscopic and macroscopic equations and analyse the degree of agreement between solutions of the macroscopic problem and solutions of the original microscopic model defined at the cell wall scale. An important contribution to the numerical simulation of the macroscopic two-scale problem is the development of a two-scale numerical algorithm that allows an efficient coupling between macroscopic and microscopic properties and processes.  Along with numerical efficiency, the advantage of the derivation of macroscopic tissue level models for plant growth allow us to study different biological settings which are difficult or even impossible to formulate and simulate at the cell-scale level.  

The paper is organized as follows. In section~\ref{section_micro} we formulate and analyse the microscopic model for plant tissue growth. The macroscopic model is derived in sections~\ref{section_deriv} using both formal asymptotic expansion and rigorous derivation applying the two-scale convergence and the periodic unfolding methods. Description of the numerical simulation algorithm and implementation of the numerical methods is given in section~\ref{section_numerics}. In section~\ref{sec:experiments} we present numerical simulation results, followed by a short conclusion and an appendix.

%%%%%%%%%%%%%%%%%%%%%%%%%%%%%%%%%%
\section{Derivation of microscopic model for plant tissue growth} \label{section_micro}
%%%%%%%%%%%%%%%%%%%%%%%%%%%%%%%%%%

In our mathematical model for the growth of a plant tissue we consider a microscopic geometry of a plant tissue composed of cells surrounded by cell walls, connected by middle lamella, see specific geometry in Figure~\ref{fig1}a-b. In modelling the dynamics of plant tissue we consider the elastic deformations and growth of these cell walls and middle lamella.

By $\Omega_t \subset \bbR^d$, with $t>0$ and $d=2$~or~$3$,  we shall represent a part of a plant tissue in the current configuration at time $t$ and $\partial \Omega_t$ denotes the external boundary of the tissue. We shall consider the growth and elastic deformation of a plant tissue given by the map $\chi(t, \cdot) : \Omega \to \mathbb R^d$ from the initial (reference)  configuration $\Omega \subset \mathbb R^d$ into deformed (current) configuration $\Omega_t = \chi(t, \Omega)$ of a plant tissue for $t>0$. We consider $\Omega $ to be a bounded Lipschitz convex or a bounded $C^{1,\tilde\gamma}$, with $\tilde \gamma \in (0, 1)$, domain.  

When modelling growth we use the framework of morphoelasticity and  consider the  multiplicative decomposition of the deformation gradient $\F = {\bf I} + \nabla \bu$ into elastic and growth parts  $\F= \F_{e}\F_g$, where $\bu$ is the displacement of the cell walls and middle lamella according to the map $\chi$ and  $\F_\re$ and $\F_\rg$ are elastic and growth deformation gradients, respectively. Then  the elastic strain is given by
\begin{equation}\label{ElasticStrain}
\begin{aligned}
{\bf E}^{\rm el} =\frac 12 (\F^T_{\rm e} \F_{\rm e} - {\bf I})
  = \frac 12 [ (\F_{\re} - {\bf I})^T(\F_{\re} - {\bf I}) +   (\F_{\re} - {\bf I}) + (\F_{\re} - {\bf I})^T ]\\
%  = \frac 12 [ (\F_{\re} - \bI)(\F_{\re} - \bI)^T + \F_{\re} + \F_{\re}^T - 2\bI] \\
  = 
  \frac 12 [ (\F_{\re} - {\bf I})^T(\F_{\re} - {\bf I}) + ({\bf I}+ \nabla \bu)\F_\rg^{-1} + (({\bf I}+\nabla \bu) \F_{\rg}^{-1})^T - 2{\bf I}] .
\end{aligned} 
\end{equation}
We model the cell walls and middle lamella as an hyperelastic material and consider the stress tensor in the form $${\boldsymbol{\sigma}}(x,\F_\re)= J_{\rm e}^{-1} \F_\re \partial_{\F_\re} W(x,\F_\re)^T,$$
where $W$ is the strain energy function and $J_{\rm e} = {\rm det} (\F_\re) = {\rm det} (\F \F_g^{-1})$. 
Then the constitutive equation for elastic deformations   is given by
$${\boldsymbol{\sigma}}(x,\F_\re)={\boldsymbol{\sigma}}(x, \nabla \bu, \F_\rg^{-1})= J_{\rm e}^{-1} \F \F_\rg^{-1} \partial_{\F_\re} W(x,\F\F_\rg^{-1})^T. $$
Mechanical equilibrium requires that   
\begin{equation}\label{elasticity_model_cur1}
 -{\rm div}_{\tilde x}\, {\boldsymbol{\sigma}}(x,\nabla \bu, \F_\rg^{-1}) = 0 \quad   \text{ in } \;  \Omega^w_t, \; t>0,
\end{equation}
where $\Omega^w_t$ denotes the domain of cell walls and middle lamellae, joining the walls of individual cells together,  and $\tilde x$ denotes the coordinates in the current configuration.
We complete equations~\eqref{elasticity_model_cur1} with the boundary conditions 
\begin{equation}\label{elasticity_model_cur}
\begin{aligned} 
{\boldsymbol{\sigma}} (x,\nabla \bu, \F_\rg^{-1})  \nu & = -  P \nu  && \text{ on } \;  \Gamma_t,   \; t>0, \\
{\boldsymbol{\sigma}} (x,\nabla \bu, \F_\rg^{-1}) \nu  &=  \bbf  && \text{ on } \; \partial \Omega_t \setminus \Gamma_{D, t} ,\; t>0, \\
   \bu \cdot \nu = 0, \; \;  \Pi_\tau \big({\boldsymbol{\sigma}} (x,\nabla \bu, \F_\rg^{-1}) \nu\big) =0  \quad & \text{ or } \quad  \bu  = 0  && \text{ on } \;  \Gamma_{D, t}, \; t>0, 
\end{aligned} 
\end{equation}
where $P$ is the turgor pressure inside the cells which can vary between cells and across the tissue, $\bbf$ denotes external forces,  $\nu$ is the external normal vector to the boundaries of $\Omega_t^w$ in current configuration,  $\Gamma_t$ denotes the boundary of cells, corresponding to plasma membranes,  $\Gamma_{D, t} \subset \partial\Omega_t$  for $t\geq 0$, and $\Pi_\tau$ denotes  the tangential components of the corresponding vector.

We can rewrite  \eqref{elasticity_model_cur1} and \eqref{elasticity_model_cur}, defined in the current configuration,  in  the reference configuration  to obtain  
\begin{equation}\label{elasticity_model_ref}
\begin{aligned} 
-{\rm div} (J_\rg \, \bS^T(x,\nabla \bu, \F_\rg^{-1}) \, \F_\rg^{-T} ) &= 0 \quad  && \text{ in } \;  \Omega^w, \; t>0, \\
J_\rg  \, \bS^T(x,\nabla \bu, \F_\rg^{-1}) \F_\rg^{-T} N & = -  J \, P\,   \F^{-T} N  && \text{ on } \;  \Gamma,   \; t>0, \\
J_\rg \,  \bS^T(x,\nabla \bu, \F_\rg^{-1}) \F_\rg^{-T}  N  &=   \bbf \,  J\,   |\F^{-T}  N|   && \text{ on } \; \partial \Omega \setminus \Gamma_D,\; t>0, \\
\bu \cdot N = 0, \; \; \Pi_\tau \big(J_\rg \,  \bS^T(x,\nabla \bu, \F_\rg^{-1}) \F_\rg^{-T}  N \big)=0 \quad & \text{ or } \quad  \bu = 0  && \text{ on } \;  \Gamma_{D}, \; t>0, 
\end{aligned} 
\end{equation}
where $N$ denotes the external normal vector  to the boundaries of the reference domain of cell walls  and middle lamella $\Omega^w\subset \mathbb R^d$, and  the nominal elastic stress is given by
$$ \bS^T(x,\nabla \bu, \F_\rg^{-1}) = \partial_{\F_\re} W(x,\F\F_\rg^{-1}).$$
 
To specify the constitutive relation for the stress tensor $\boldsymbol{\sigma}$, or correspondent nominal tensor $\bS$, we assume small elastic strain (small elastic deformations of plant tissues), i.e.
$$
 \boldsymbol{\sigma}(x,\nabla \bu, \F_\rg^{-1}) \approx \bS^T (x,\nabla \bu, \F_\rg^{-1})
\approx{\mathbb E}(x) {\boldsymbol{\ve}}^{\rm el}(\nabla \bu, \F_\rg^{-1}), 
$$
where $\bbE(x) = \partial_{\F_\re} \partial_{\F_\re} W(x,{\bf 1})$ is the elasticity tensor and $\boldsymbol{\ve}^{\rm el}$ is the linearised version of the elastic strain \eqref{ElasticStrain}, which depends on the displacement gradient  and growth tensor
$$
{\boldsymbol{\ve}}^{\rm el}(\nabla \bu, \F_\rg^{-1}) = \frac 12 \big[ \nabla \bu\,   \F_\rg^{-1}  + (\nabla \bu\,  \F_\rg^{-1})^T  + \F_\rg^{-1} + \F_\rg^{-T} - 2 \bI \big] ={\rm sym}(\nabla \bu\,   \F_\rg^{-1}) + {\rm sym}(\F_\rg^{-1}) - \bI  .
$$
We also assume 
$J \F^{-T} \approx J_g \F_\rg^{-T}$ for small elastic strain. 
 Notice that for   $\F_\rg = {\bf I}$ we recover the standard formula for the strain in the case of linear elasticity.  

To complete the model we specify equation for the growth tensor $\F_\rg$
\begin{equation}\label{growth_2}
\begin{aligned} 
& \frac{\partial  \F_\rg}{\partial t}  =   G(x, \nabla \bu, \F_\rg^{-1} ) \, \F_\rg && \text{ in } \; \Omega^w, \; \; t>0,  \\
&\F_\rg(0,x) = {\bf I} && \text{ for } \, x\in  \Omega^w.
\end{aligned} 
\end{equation}
Since there is no consensus on modelling growth~\cite{Goriely_2017,Smithers_2019}, we consider two scenarii and assume that the growth depends on the local average of  stress or of strain  in cell walls and middle lamella (both are compatible with Lockhart's law~\cite{Lockhart}). We also assume that the cell wall and middle lamella expand when the local average of the stress or strain is larger than some threshold value~\cite{Lockhart}. Hence we consider the stress based growth
$$
\tilde G(x, \nabla \bu, \F_\rg^{-1} ) =
\tilde G(x,  \boldsymbol{\sigma}(x,\nabla \bu, \F_\rg^{-1})  ) =
 \eta_\sigma[ {\boldsymbol{\hat{\sigma}}} - \boldsymbol{\tau}_\sigma ]_+ ,
$$
or the elastic strain based growth
$$
\tilde G(x, \nabla \bu, \F_\rg^{-1} ) =
 \tilde G(x,  {\boldsymbol{\ve}}^{\rm el}(\nabla \bu, \F_\rg^{-1}) ) =
 \eta_\ve[\boldsymbol{\hat{\ve}}^{\rm el} - {\boldsymbol\tau}_\ve ]_+ , 
$$ 
and 
\begin{equation}\label{def_G_11}
G(x, \nabla \bu, \F_\rg^{-1} )_{ij} = 
\begin{cases} 
-M & \text{ if } \; \tilde G(x, \nabla \bu, \F_\rg^{-1} )_{ij}  \leq -M , \\
\tilde G(x, \nabla \bu, \F_\rg^{-1} )_{ij} & \text{ if }  \; 
-M < \tilde G(x, \nabla \bu, \F_\rg^{-1} )_{ij} < M , \\
M & \text{ if } \; \tilde G(x, \nabla \bu, \F_\rg^{-1} )_{ij} \geq M, 
\end{cases} 
\quad\text{for  } \; i,j=1, \ldots, d, 
\end{equation}
for some $M>0$. 
Here 
%$[v]_+ = (\max\{ v_{ij}, 0\})_{i,j=1, \ldots, d}$.  
$[v]_+=Q^T[Dv]^+Q$ with $Q$ the rotation that diagonalizes the tensor $v$, and $[Dv]^+$ is a diagonal matrix with the positive parts of the eigenvalues of $v$ on its diagonal. This allows to relate growth only to tensile stress, in the case of stress based growth, or only to elongational strain, in the case of strain based growth. The uniform boundedness assumption on the growth rate $G$ is used in the rigorous analysis of the model and is not restrictive from the biological point of view. 
Notice that in the growth laws we consider piece-wise constant fields
$\boldsymbol{\hat{\sigma}}$ and $\boldsymbol{\hat{\ve}}^{\rm el}$, obtained as an average over each cell of $\boldsymbol{\sigma}$ and $\boldsymbol{\ve}^{\rm el}$ respectively. 
The growth laws introduce two  physical parameters: the extensibility constants $\eta_\sigma$ and $\eta_\ve$ and  threshold matrices ${\boldsymbol\tau}_\sigma$ and ${\boldsymbol\tau}_\ve$, respectively.

\subsection{Formulation of the microscopic model}

We  assume that in a plant tissue  cells are distributed periodically and consider  the parameter~$\delta >0$ that determines the ratio between the size of a cell and the size of the tissue. We  also assume that the size of the cell and the thickness of the cell wall~$h$ are of the same order and much smaller than the size of the tissue, i.e.~$\delta$ is small.
To define  the microscopic structure of the plant tissue given by the cell walls and middle lamella, we consider a `unit cell' $Y$ and $Y_c\subset Y$ represents the  voids filled by biological cells,  with Lipschitz boundary~$\Gamma=\partial Y_c$ and composed of a finite number of subdomains separated from each other and from the edges of $Y$, whereas $Y_w = Y \setminus \overline Y_c$ represents cell walls surrounded by middle lamella. Then the  microscopic geometry of a plant tissue in the reference configuration  is defined as
$$
\Omega_{c}^\delta = \bigcup_{\xi \in \Xi_\delta} \delta\, (\overline Y_c + \xi)   \; \; \text{ and } \; \; \Omega^\delta =  \Omega \setminus \Omega_{c}^\delta,
$$
where   $\Xi_\delta = \{  \xi \in \Xi   :      \delta\, (\overline Y + \xi) \subset \overline \Omega, \, {\rm dist}(\delta\, (\overline Y_c + \xi), \partial \Omega) \geq \kappa \delta\}$ and $\Xi = \{ \xi \in \mathbb R^d  :  \xi = \sum_{j=1}^d k_j b_j, \,   k \in \mathbb Z^d \}$, with $\{b_j\}_{j=1}^d$ being the basis vectors of $Y$, i.e.~$Y = \{y \in \mathbb R^d : y = \sum_{j=1}^d s_j b_j, \, s \in (0,1)^d\}$, and for some fixed $\kappa >0$.  The boundaries of  $\Omega^\delta$ that correspond to cell plasma-membranes  are
$$
\Gamma^\delta = \bigcup_{\xi \in \Xi_\delta} \delta \, (\Gamma + \xi)
\quad \text{ and } \quad  \Gamma^\delta = \partial \Omega^\delta \setminus \partial \Omega.
$$
The parts of the tissues near the boundary $\partial \Omega$, which do not include the complete $\delta \, \overline Y$ are denoted by
 $$\Lambda_\delta = \Omega \setminus \bigcup_{\xi \in \Xi_\delta} \delta \, (\overline Y + \xi)\;  \text{ and } \;
 \Lambda_\delta  = \bigcup_{l=1}^L \Lambda^l_\delta, \; \text{ with } \;   \Lambda^l_\delta \cap \delta \, (Y + \xi) \neq \emptyset  \; \text{ for one } \;  \xi \in  \Xi \; \text{ and } \; L= \mathcal O(1/\delta^{d-1}).
 $$
Then  microscopic  equations for elastic deformations  in the reference domain   read  
\begin{equation}\label{micro_model_ref}
\begin{aligned} 
-{\rm div} (J_\rg^\delta\,  \bS^T_\delta(x,  \nabla \bu^\delta, \F_{\rg, \delta}^{-1})\, \F_{\rg, \delta}^{-T} ) &= 0 \quad  && \text{ in } \; \Omega^\delta, \; t>0, \\
J_\rg^\delta \,  \bS_\delta^T (x,  \nabla \bu^\delta, \F_{\rg, \delta}^{-1})\, \F_{\rg, \delta}^{-T} N & = -   J_\rg^\delta \,  P^\delta(t,x) \F_{\rg, \delta}^{-T} N  && \text{ on } \;  \Gamma^\delta,   \; t>0, \\
J_\rg^\delta\,  \bS_\delta^T (x, \nabla \bu^\delta, \F_{\rg, \delta}^{-1})\,  \F_{\rg, \delta}^{-T}  N  &=  J_\rg^\delta\,  \bbf(t,x) \,   |\F_{\rg, \delta}^{-T}  N|   && \text{ on } \; \partial \Omega \setminus \Gamma_D,\; t>0,  \\
\bu^\delta \cdot N=0, \; \; \Pi_\tau\big(J_\rg^\delta\,  \bS_\delta^T (x, \nabla \bu^\delta, \F_{\rg, \delta}^{-1})\,  \F_{\rg, \delta}^{-T}  N \big) =0 \quad & \text{ or } \quad  \bu^\delta = 0  && \text{ on } \; \Gamma_D,\; t>0, 
\end{aligned} 
\end{equation}
where  $J_\rg^\delta = {\rm det} (\F_{\rg, \delta})$,  $\bS_\delta^T (x,  \nabla \bu^\delta, \F_{\rg, \delta}^{-1}) = \mathbb E^\delta(x) {\boldsymbol\ve}^{\rm el}(\nabla \bu^\delta, \F_{\rg, \delta}^{-1}) =  \mathbb  E^\delta(x) [{\rm sym}(\nabla \bu^\delta \F_{\rg, \delta}^{-1})  + {\rm sym}(\F_{\rg, \delta}^{-1})  -  \bI]$, with $\mathbb E^\delta(x) = \mathbb E(x, x/\delta)$,  and $P^\delta(t,x) = P(t,x, x/\delta)$ for given functions $\mathbb E: \Omega\times Y \to \mathbb R^{4 d}$ and $P: (0,T)\times \Omega\times \Gamma \to \mathbb R$, extended  $Y$-periodically to $\mathbb R^d$ and to $\bigcup_{\xi \in \Xi}(\Gamma + \xi)$ respectively.

The dynamics of the growth tensor is determined by  
\begin{equation}\label{growth_micro_1}
\begin{aligned} 
& \frac{\partial  \F_{\rg, \delta} }{\partial t}  =   G^\delta(x, \nabla \bu^\delta, \F_{\rg, \delta}^{-1}) \, \F_{\rg, \delta} \quad &&   \text{ in } \; \Omega^\delta, \; \; t>0,  \\
&\F_{\rg, \delta}(0,x) = \bI &&   \text{ in } \; \Omega^\delta,
\end{aligned} 
\end{equation}
 where 
$$
\begin{aligned}
\tilde G^\delta(x, \nabla \bu^\delta, \F_{\rg, \delta}^{-1}) &= \eta_{\sigma} \Big[\ddashinttt_{\delta( [x/\delta]_{Y}+  Y_w)\cap \Omega} {\boldsymbol\sigma}(\tilde x, \nabla\bu^\delta, \F_{\rg, \delta}^{-1} )\,  d\tilde x - {\boldsymbol \tau}_\sigma \Big]_{+}  & \qquad \text{  or  } \\
\tilde G^\delta(x, \nabla \bu^\delta, \F_{\rg, \delta}^{-1})& = \eta_{\ve} \Big[\ddashinttt_{\delta( [x/\delta]_{Y}+  Y_w)\cap \Omega} {\boldsymbol\ve}^{\rm el}(\nabla \bu^\delta, \F_{\rg, \delta}^{-1}) \, d\tilde x - {\boldsymbol \tau}_\ve \Big]_{+},
\end{aligned}
$$
with ${\boldsymbol\sigma}(x, \nabla\bu^\delta, \F_{\rg, \delta}^{-1} ) = \bS^T_\delta (x,  \nabla \bu^\delta, \F_{\rg, \delta}^{-1})= \mathbb E^\delta(x) {\boldsymbol\ve}^{\rm el}(\nabla \bu^\delta, \F_{\rg, \delta}^{-1})$, and
\begin{equation}\label{def_G}
G^\delta(x, \nabla \bu^\delta, \F_{\rg, \delta}^{-1})_{ij} = 
\begin{cases}
- M & \text{ if } \;  \qquad \quad \tilde G^\delta(x, \nabla \bu^\delta, \F_{\rg, \delta}^{-1})_{ij} \leq - M,\\
\tilde G^\delta(x, \nabla \bu^\delta, \F_{\rg, \delta}^{-1})_{ij}   & \text{ if } \; - M < \tilde G^\delta(x, \nabla \bu^\delta, \F_{\rg, \delta}^{-1})_{ij} < M ,  \\
M & \text{ if } \; \qquad \quad  \tilde G^\delta(x, \nabla \bu^\delta, \F_{\rg, \delta}^{-1})_{ij} \geq M,
\end{cases} 
\qquad \text{ for }  i,j=1, \ldots, d.
\end{equation}

\subsection{Well-posedness of microscopic model} 

To prove existence of a weak solution of  problem \eqref{micro_model_ref}  and \eqref{growth_micro_1}, 
we consider standard ellipticity assumptions on the elasticity tensor $\mathbb E$ and regularity assumptions on the pressure $P^\delta$ and boundary forces ${\bf f}$. 
For the pressure  inside of the cells $P^\delta$ we shall consider dependence on the microscopic structure in the  form
\begin{equation}\label{form_pressure}
P^\delta(t,x) = P_1(t,x) + \delta P_2(t,x, x/\delta), 
\end{equation}
for some given functions $P_1: (0,T)\times\Omega \to \mathbb R$ and $P_2: (0,T)\times \Omega \times \Gamma \to \mathbb R$. 
\begin{assumption} \label{assumption}
\begin{itemize} 
\item Elasticity tensor $\mathbb E \in C^\gamma(\overline \Omega; L^\infty_{\rm per}(Y))^{4d}$ is positive definite  and bounded, i.e.~$\alpha_1 |\bA|^2 \leq \mathbb E(x, y) \bA \cdot \bA \leq \alpha_2 |\bA|^2$  for $x\in \Omega$, $y \in Y$, symmetric matrices ${\bf A} \in \mathbb R^{d\times d}$, and  positive constants $\alpha_1, \alpha_2$, and has minor and major symmetries, i.e.\ $\mathbb E_{ijkl} = \mathbb E_{klij} = \mathbb E_{jikl} = \mathbb E_{ijlk}$, for $i,j,k,l=1, \ldots, d$, \; for $\gamma\in (0,1)$. 
%\item Function $ G: \mathbb R^d\times \mathbb R^{d\times d} \to \mathbb R^d$ is Lipschitz continuous and bounded. 
\item ${\bf f} \in C^\gamma([0,T]\times \partial \Omega)$, $P_1 \in C^\gamma([0,T]; C^1(\overline \Omega))$ and   $P_2 \in C^\gamma([0,T]\times \overline \Omega \times \Gamma)$, where $P_2(t, x, \cdot)$ is $Y$-periodically extended to $\Gamma + \xi$, with $\xi \in \Xi$, for $\gamma \in (0,1)$.
\end{itemize}
\end{assumption}

If $\Gamma_D \neq \emptyset$, consider $\mathcal G$ the symmetry group of $\Gamma_D$, formed of at most one, for $d=2$, or two, for $d=3$, translations and at most one rotation, for $d=3$, that leave $\Gamma_D$ invariant. Let $\mathcal G_\tau$ be the subspace of ${\mathbb R}^d$ spanned by the set of translations in $\mathcal G$ and $\{\boldsymbol{\rho}_1,\boldsymbol{\rho}_2\}$ be the orthonormal basis of the plane perpendicular to the rotation axis. Then, depending on the boundary conditions, we define the following space for solutions of~\eqref{micro_model_ref}
$$
\begin{aligned}
V_\delta  = & \{ \bu  \in H^1(\Omega^\delta)^d \, :   \; \bu = 0  \, \text{ on } \, \Gamma_D  \}  \qquad \qquad \text{ or } \\
V_\delta  = & \{ \bu  \in H^1(\Omega^\delta)^d \, : \;  \bu\cdot N = 0 \, \text{ on } \, \Gamma_D, \;  \int_{\Omega^\delta}  \Pi_{\mathcal G_\tau} (\bu) \, dx =0, \,   \, \int_{\Omega} \boldsymbol{\rho}_1(\nabla \hat\bu - \nabla \hat \bu^T)\boldsymbol{\rho}_2\,  dx=0 \},
\end{aligned}
$$ 
where 
$\Pi_{\mathcal G_\tau}$ is the projection on $\mathcal G_\tau$ and  $\hat \bu$ denotes an extension of $\bu$ from $\Omega^\delta$ into $\Omega$, see e.g.~\cite{OSY}.
If $\Gamma_D = \emptyset$, 
then 
$$V_\delta=\{\bu  \in H^1(\Omega^\delta)^d \; : \int_{\Omega^\delta} \bu \, dx =0, \;  \int_{\Omega} \big(\partial_{x_j} \hat \bu_i - \partial_{x_i} \hat \bu_j\big) dx =0, \;  \text{ for } i\neq j, \; i,j=1, \ldots, d \}. $$

Using  \eqref{form_pressure} for $P^\delta$ and  ${\rm div} (J_\rg^\delta \F_{\rg, \delta}^{-T}) = 0$ in $\Omega^\delta$ we can rewrite equations in \eqref{micro_model_ref} as  
\begin{equation}\label{micro_model_ref_2}
\begin{aligned} 
-{\rm div} \Big(J_\rg^\delta\,  \Big(\bS^T_\delta(x,  \nabla \bu^\delta, \F_{\rg, \delta}^{-1}) + P_1(t,x) {\bf I}\Big)\, \F_{\rg, \delta}^{-T} \Big)    &
= - J_\rg^\delta \F_{\rg, \delta}^{-T}  \nabla P_1(t,x)  \quad  && \text{ in }  \Omega^\delta, \; t>0, \\
J_\rg^\delta \, \Big( \bS_\delta^T (x,  \nabla \bu^\delta, \F_{\rg, \delta}^{-1}) + P_1(t,x) {\bf I} \Big)\, \F_{\rg, \delta}^{-T} N & = -  \delta \,  J_\rg^\delta \,  P_2^\delta(t,x) \F_{\rg, \delta}^{-T} N  && \text{ on }   \Gamma^\delta,  \; t>0,   \\
J_\rg^\delta\,  \bS_\delta^T (x, \nabla \bu^\delta, \F_{\rg, \delta}^{-1}) \,  \F_{\rg, \delta}^{-T}  N  &= J_\rg^\delta\,  \bbf(t,x)    |\F_{\rg, \delta}^{-T}  N|   && \text{ on }   \Gamma_N, \; t>0, \\
\bu^\delta \cdot N =0, \;  J_\rg^\delta\, \Pi_\tau\big( \bS_\delta^T (x, \nabla \bu^\delta, \F_{\rg, \delta}^{-1})\,  \F_{\rg, \delta}^{-T}  N \big) = 0 \quad  & \text{ or } \quad \bu^\delta = 0  && \text{ on }  \Gamma_D,\; t>0, 
\end{aligned} 
\end{equation} 
where $\Gamma_N = \partial \Omega \setminus \Gamma_D$ and $P_2^\delta(t,x) = P_2(t,x, x/\delta)$. 

In the analysis and numerical implementation of   \eqref{growth_micro_1} and \eqref{micro_model_ref_2}  we shall consider weak solutions of the model equations. We shall use the notation $\langle \phi, \psi\rangle_A = \int_A \phi \psi\, dx$, for $\phi \in L^p(A)$, $\psi \in L^q(A)$,  and
$\langle \phi, \psi\rangle_{\partial A} = \int_{\partial A} \phi \psi\, d\gamma$, for $\phi \in L^p(\partial A)$, $\psi \in L^q(\partial A)$, with $1< p, q < \infty$, $1/p+1/q =1$, and a bounded Lipschitz domain $A$.

\begin{definition} 
A weak solution of \eqref{growth_micro_1}  and \eqref{micro_model_ref_2}   are  $\bu^\delta \in L^2(0,T; V_\delta)$ and $\F_{\rg, \delta} \in W^{1, \infty}(0,T;  L^q (\Omega^\delta))^{d\times d}$, for any $q \in (1,\infty)$, with $\F_{\rg, \delta} \in   L^\infty ((0,T)\times \Omega^\delta)^{d\times d}$, satisfying
\begin{equation} \label{var_inequal}
\begin{aligned} 
 \Big\langle  J_\rg^\delta \,  \mathbb E^\delta(x) {\rm sym}(\nabla \bu^\delta \F_{\rg, \delta}^{-1}) + P_1(t,x) {\bf I} ,  \nabla \varphi \F_{\rg, \delta}^{-1} \Big\rangle_{\Omega^\delta} + \delta \, \Big \langle J_{\rg}^\delta \, P_2^\delta(t,x)\, \F_{\rg, \delta}^{-T}   N,  \varphi \Big \rangle_{\Gamma^\delta}\\
   - \Big \langle J_{\rg}^\delta \, P_1(t,x)\, \F_{\rg, \delta}^{-T}   N,  \varphi \Big \rangle_{\partial \Omega}  = \Big
\langle J_\rg^\delta\,  \mathbb E^\delta(x) \big[\bI -{\rm sym} (\F_{\rg, \delta}^{-1})\big],   \nabla \varphi \F_{\rg, \delta}^{-1} \Big\rangle_{\Omega^\delta} 
\\    - \Big
\langle J_\rg^\delta\, \F_{\rg, \delta}^{-T} \nabla P_1(t,x), \varphi \Big\rangle_{\Omega^\delta} 
+ \Big\langle  J_\rg^\delta\, \bbf(t,x) \, |\F_{\rg, \delta}^{-T} N|, \varphi \Big\rangle_{\Gamma_N} 
\end{aligned} 
\end{equation} 
for $\varphi \in L^2(0,T; V_\delta)$ and a.a.~$t \in (0,T)$, and $\F_{\rg, \delta}$ satisfies \eqref{growth_micro_1}  a.e.~in $(0,T)\times\Omega^\delta$.
\end{definition}

First we shall prove a version of the Korn inequality, where the symmetric gradient  includes the growth tensor. 
\begin{lemma}\label{Korn_gener}
For $\bu \in V_\delta$ and   a tensor  $\F \in L^\infty(\Omega)$ such that  on each $\delta (Y + \xi)\cap \Omega$, for $\xi \in \Xi$, $\F$ is constant, ${\rm det}(\F)\geq 1$ and eigenvalues   $\lambda_j(\F) \geq 1$ for $j=1,\ldots, d$,  we have the following estimate
\begin{equation}\label{estim:Korn_gen}
\|\bu \|^2_{L^2(\Omega^\delta)}+ \|\nabla \bu\|^2_{L^2(\Omega^\delta)} \leq  C \|{\rm{sym}}\big(\nabla \bu \F^{-1} \big)\|^2_{L^2(\Omega^\delta)}, 
\end{equation}
with constant  $C = C(\F)$  independent of $\delta$. 
\end{lemma}
\begin{proof}
Assumptions on the microscopic structure of the plant tissues ensure the existence of an extension $\hat \bu$ of $\bu$ from $\Omega^\delta$ to $\Omega$ with 
$$
\begin{aligned}
 \|\hat \bu \|_{H^1(\Omega)} & \leq C \| \bu\|_{H^1(\Omega^\delta)}, \quad &
 \|\hat \bu \|_{L^2(\Omega)} + \|{\rm{sym}}(\nabla \hat \bu) \|_{L^2(\Omega)} & \leq C \big(\| \bu\|_{L^2(\Omega^\delta)} + \| {\rm{sym}} (\nabla \bu )\|_{L^2(\Omega^\delta)}\big), \\
 \|\nabla \hat \bu \|_{L^2(\Omega)} & \leq C \| \nabla \bu\|_{L^2(\Omega^\delta)},  \quad &
 \|{\rm{sym}}(\nabla \hat \bu) \|_{L^2(\Omega)} &\leq C \| {\rm{sym}} (\nabla \bu )\|_{L^2(\Omega^\delta)},
\end{aligned}
$$
 where constant $C$ is independent of $\delta$, see e.g.~\cite{OSY}.  In the following we shall identify $\bu$ with its extension $\hat \bu$. Using that  $\F$  and hence   $\F^{-T}$  are  constant on each $\delta (Y + \xi)\cap \Omega$, for $\xi \in \Xi$, and the properties of $\Omega^\delta$, together with an extension of $\bu$ from $ \delta \F(\xi)Y_w$ into $ \delta \F(\xi)Y$, constructed as in e.g.~\cite[Lemma~4.1]{OSY},  we obtain
$$
\begin{aligned} 
& \int_{\Omega^\delta} |{\rm{sym}}(\nabla \bu \F^{-1})|^2 dx = 
\sum_{\xi \in \Xi_\delta} \int_{\delta (Y_w + \xi)}
|{\rm{sym}}(\nabla \bu \F^{-1})|^2 dx
+ \sum_{l=1}^L\int_{\Lambda_\delta^l} |{\rm{sym}}(\nabla \bu \F^{-1})|^2 dx\\
&\quad  = \sum_{\xi \in \Xi_\delta} {\rm det } (\F^{-1}(\xi))\int_{\delta \F(\xi)(Y_w + \xi)}
|{\rm{sym}}(\nabla \bu ) |^2 dx
+ \sum_{l=1}^L
{\rm det}(\F^{-1}(\xi_l))
\int_{\F(\xi_l)\Lambda_\delta^l} |{\rm{sym}}(\nabla \bu)|^2 dx \\
& \quad \geq  \sum_{\xi \in \Xi_\delta} {\rm det } (\F^{-1}(\xi)) C(\F(\xi))\int_{\delta \F(\xi)(Y + \xi)}
|{\rm{sym}}(\nabla \bu) |^2 dx
+ \sum_{l=1}^L
{\rm det}(\F^{-1}(\xi_l))
\int_{\F(\xi_l)\Lambda_\delta^l} |{\rm{sym}}(\nabla \bu )|^2 dx,
\end{aligned}
$$
where  the dependence of the constant $C$  on $\F(\xi)$, for $\xi \in \Xi_\delta$, arises from the application of the  Korn and Poincar\'e inequalities when constructing an extension from $\F(\xi)Y_w$ into $\F(\xi)Y$, see \cite[Lemma 4.1]{OSY} for more details, and $\xi_l \in \Lambda_\delta^l$ for $l=1, \ldots, L$.
Using the uniform boundedness of $\F$, together with the fact that   ${\rm det}(\F)\geq 1$ in $(0,T)\times \Omega$ and eigenvalues $\lambda_j(\F) \geq 1$, with $j=1,\ldots, d$,
and  applying the Korn inequality, see e.g.~\cite{Ciarlet, Lions_1976, OSY},  yield
$$
\begin{aligned} 
\int_{\Omega^\delta} |{\rm sym} (\nabla \bu \F^{-1})|^2 dx &\geq C \int_{\Omega_F}  
|{\rm sym}(\nabla \bu)  |^2   dx
\geq C \int_{\Omega}  
|{\rm sym} (\nabla \bu)  |^2   dx \\
&\geq C \int_{\Omega} \big(| \bu |^2 + |\nabla \bu |^2   \big)dx 
\geq  C \int_{\Omega^\delta} (\big |\bu |^2  +|\nabla \bu |^2 \big)  dx , 
\end{aligned}
$$
for $\bu \in V_\delta$, where $\Omega_F = \big(\bigcup_{\xi \in \Xi_\delta} \delta \F(\xi)(Y + \xi)\big) \bigcup \big(\bigcup_{l=1}^L \F(\xi_l) \Lambda_\delta^l\big) $. This implies the result stated in the lemma.
\end{proof}
\begin{remark} 
Some results on the generalisation of the  Korn inequality can be found in~\cite{Neff_2002, Pompe_2003}. Notice that in \cite{Neff_2002} $C^1$-regularity of $\F$ is required and in \cite{Pompe_2003} the continuity of $\F$ or $\bu=0$ on $\partial \Omega$ are assumed. In the proof of Lemma~\ref{Korn_gener} we use the fact that $\F$ is uniformly bounded and piece-wise constant and the eigenvalues of $\F$ satisfy   $\lambda_j(\F) \geq 1$, for $j=1,\ldots, d$, ensuring that $\Omega$ is a subdomain of the transformed domain $\Omega_F$. 
\end{remark}

Using the inequality \eqref{estim:Korn_gen} and applying the Banach fixed point theorem we prove the well-posedness result for microscopic model \eqref{growth_micro_1} and \eqref{micro_model_ref_2}. 

\begin{theorem} \label{th:exist}
Under Assumptions~\ref{assumption} there exists a unique  weak solution of   \eqref{growth_micro_1}, \eqref{micro_model_ref_2} satisfying 
\begin{equation} \label{estim_apriori_1}
\|\bu^\delta \|_{L^\infty(0,T; V_\delta)} + \|\F_{\rm g, \delta} \|_{L^\infty((0,T)\times\Omega^\delta)}  +
\|\partial_t \F_{\rm g, \delta} \|_{L^\infty((0,T)\times \Omega^\delta)} + \|\F_{\rm g, \delta} \|_{W^{1,\infty}(0,T; L^q(\Omega^\delta))} \leq C,
\end{equation}
with a constant $C$ independent of $\delta$ and any $q \in (1, \infty)$.
\end{theorem} 
\begin{proof}  We  shall apply a fixed point argument to show existence of a solution of     \eqref{growth_micro_1}, \eqref{micro_model_ref_2}. Consider
$$
\begin{aligned}
\mathcal W = & \{ \F \in L^\infty ((0,T)\times\Omega^\delta)^{d\times d} :  \F \in W^{1,\infty}(0,T; L^q(\Omega^\delta))^{d\times d}, \;  \|\F\|_{L^\infty((0,T)\times\Omega^\delta)} \leq \hat\kappa, \; {\rm det}(\F) \geq 1 , \; \\ & \lambda_j(\F) \geq 1 \text{ in } [0, T]\times \Omega^\delta,  j=1, \ldots d,
\text{ and piece-wise constant  in each } \delta (Y_w + \xi)\cap \Omega \; \text{ for } \; \xi \in \Xi\},
\end{aligned}
$$ 
with $\hat\kappa \geq  \exp(d MT)$. For
a given  $\tilde \F_{\rg, \delta} \in \mathcal W$,
taking $\bu^\delta$ as a test function in \eqref{var_inequal} and using assumptions on $\mathbb E$, $P_1$, $P_2$, and $\bbf$, together with the fact that $1 \leq  {\rm det} (\tilde \F_{\rg, \delta}) \leq C$, with a constant $C$ independent of $\delta$,  we obtain 
\begin{equation}\label{estim_u_11}
 \|{\rm{sym}}(\nabla \bu^\delta \tilde \F_{\rg, \delta}^{-1})\|^2_{L^\infty(0,T; L^2(\Omega^\delta))} \leq C\big(1 + \varsigma \|\bu^\delta \|^2_{L^\infty(0, T; V_\delta)}\big) ,
\end{equation}
for any fixed $\varsigma >0$. Here we also used the trace estimate 
$$
\delta\|v\|^2_{L^2(\Gamma^\delta)} \leq C \big( 
\|v\|^2_{L^2(\Omega^\delta)} + \delta^2 \|\nabla v\|^2_{L^2(\Omega^\delta)} \big)  \quad  \text{ for } \; v \in H^1(\Omega^\delta),
$$
for some constant $C$ independent of $\delta$, see e.g.~\cite{HJ_91}.
Notice that since $\tilde \F_{\rg, \delta}$ is constant on each $\delta(Y_w + \xi) \cap \Omega$ it can be trivially extended by the constant into each $\delta(Y+ \xi)\cap \Omega$, for $\xi \in \Xi$. Using
Lemma~\ref{Korn_gener} we obtain 
\begin{equation}\label{estim_u_delta} 
 \| \bu^\delta \|_{L^\infty(0, T; V_\delta)}
 \leq C  \|{\rm{sym}}(\nabla \bu^\delta \tilde \F_{\rg, \delta}^{-1})\|_{L^\infty(0,T; L^2(\Omega^\delta))}, 
 \end{equation}
 and then choosing in \eqref{estim_u_11}  sufficiently small $\varsigma>0$ yields
 \begin{equation}\label{estim-u-delta_1}
 \| \bu^\delta \|_{L^\infty(0, T; V_\delta)}
 \leq C, 
 \end{equation}
 where the constant  $C$ does not depend on $\delta$.  Assumptions on $\mathbb E$, $P_1, P_2$, ${\bf f}$ and $ \tilde\F_{\rg, \delta}$, together with the estimate in Lemma~\ref{Korn_gener}, ensure that  
 $$
 \begin{aligned} 
 F(\varphi) =&  
  \Big \langle J_{\rg}^\delta \, P_1(t,x)\, \tilde \F_{\rg, \delta}^{-T}   N,  \varphi \Big \rangle_{\partial \Omega}- \Big\langle P_1(t,x) {\bf I} ,  \nabla \varphi \, \tilde \F_{\rg, \delta}^{-1} \Big\rangle_{\Omega^\delta} - \delta\Big \langle J_{\rg}^\delta \, P_2^\delta(t,x)\, \tilde \F_{\rg, \delta}^{-T}   N,  \varphi \Big \rangle_{\Gamma^\delta} \\
&   + \Big
\langle J_\rg^\delta\,  \mathbb E^\delta(x) \big[\bI -{\rm sym} (\tilde \F_{\rg, \delta}^{-1})\big],   \nabla \varphi \, \tilde \F_{\rg, \delta}^{-1} \Big\rangle_{\Omega^\delta} 
- \Big
\langle J_\rg^\delta\, \tilde \F_{\rg, \delta}^{-T} \nabla P_1(t,x), \varphi \Big\rangle_{\Omega^\delta} 
+ \Big\langle  J_\rg^\delta\, \bbf(t,x) \, |\tilde \F_{\rg, \delta}^{-T} N|, \varphi \Big\rangle_{\Gamma_N} 
\end{aligned} 
 $$
 defines a bounded linear functional on $V_\delta$ and 
 $$
 B(\bu^\delta, \varphi)= \Big\langle  J_\rg^\delta \,  \mathbb E^\delta(x) {\rm sym}(\nabla \bu^\delta \tilde \F_{\rg, \delta}^{-1}),  {\rm sym} (\nabla \varphi \tilde \F_{\rg, \delta}^{-1}) \Big\rangle_{\Omega^\delta}
 $$
 is a coercive bilinear form on $V_\delta\times V_\delta$, for $t \in (0,T)$ and a given  $\tilde \F_{\rg, \delta} \in \mathcal W$, where $J_\rg^\delta = {\rm det} (\tilde \F_{\rg, \delta})$. 
Thus  the Lax-Milgram  theorem, see e.g.~\cite{Evans}, yields existence of a unique solution of problem~\eqref{micro_model_ref_2}  in $L^\infty(0,T;  V_\delta)$ for a given $\tilde \F_{\rg, \delta} \in \mathcal W$ and each fixed $\delta>0$. The boundedness of function $G^\delta$ also ensures existence of a solution of~\eqref{growth_micro_1} with  $\tilde \F_{\rg, \delta}$  instead of $\F_{\rg, \delta}$ in $G^\delta$.

To show existence of a unique solution of the full model \eqref{growth_micro_1} and \eqref{micro_model_ref_2}  we need to show a contraction property of the map $\mathcal K: \mathcal W \to \mathcal W$, where  $\F_{\rg, \delta} = \mathcal K(\tilde \F_{\rg, \delta})$ is a solution of problem  \eqref{micro_model_ref_2} and \eqref{growth_micro_1}, 
with $\tilde \F_{\rg, \delta}$ instead of $\F_{\rg, \delta}$ in equation \eqref{micro_model_ref_2} and in function $G^\delta$ in \eqref{growth_micro_1}. 

From equations for $\F_{\rg, \delta}$, using  properties of  $G^\delta$, we obtain   
$$
{\rm det} (\F_{\rg, \delta}(t,x))= {\rm det }(\F_{\rg, \delta}(0,x)) \, {\rm det }\Big( \exp\Big( \int_0^t G^\delta(x, \nabla \bu^\delta, \tilde \F_{\rg, \delta}^{-1})\,ds\Big)\Big) =\exp\Big( {\rm tr}\Big( \int_0^t G^\delta(x, \nabla \bu^\delta, \tilde \F_{\rg, \delta}^{-1})\, ds\Big)\Big) \geq 1,
$$
for $(t,x) \in [0,T]\times \overline{\Omega^\delta}$, since diagonal elements of $G^\delta$ are nonnegative.   Assumptions on $G^\delta$  imply  that $\F_{\rg, \delta}$ is piece-wise constant in each $\delta (Y_w+\xi)\cap \Omega$  for $\xi \in \Xi$, has eigenvalues greater than or equal to $1$,  and
\begin{equation}\label{bound_F_delta}
%\begin{aligned}
\|\F_{\rm g, \delta} \|_{L^\infty((0,T)\times \Omega^\delta)} + \|\partial_t \F_{\rm g, \delta} \|_{L^\infty((0,T)\times\Omega^\delta)}+ \|\F_{\rm g, \delta} \|_{W^{1,\infty}(0,T; L^q(\Omega^\delta))}  \leq C,
%\\&\|\F_{\rm g, \delta} \|_{L^\infty(\Omega^\delta_T)} = \Big\|\F_{\rm g, \delta}(0)  \exp\Big( \int_0^t G(x, \nabla \bu^\delta)ds\Big)\Big\|_{L^\infty(\Omega^\delta_T)} \geq 1,
%\end{aligned}
\end{equation}
with a constant $C$ independent of $\delta$. 
Since $J_g^\delta={\rm det}(\F_{\rm g, \delta}) \geq 1$ for $(t,x) \in [0,T]\times \overline{\Omega^\delta}$, we  also have
$$
\|\F^{-1}_{\rm g, \delta} \|_{ L^\infty((0,T)\times \Omega^\delta)} + \|\partial_t \F^{-1}_{\rm g, \delta} \|_{L^\infty((0,T)\times\Omega^\delta)} + \|\F^{-1}_{\rm g, \delta} \|_{W^{1,\infty}(0,T; L^q(\Omega^\delta))} \leq C.
$$

 Multiplying  the difference of \eqref{growth_micro_1} for $\F_{\rg, \delta}^1$ and $\F_{\rg, \delta}^2$ by $\F_{\rg, \delta}^1- \F_{\rg, \delta}^2$,  integrating over time variable,  as well as
 using the boundedness of $\tilde \F_{\rm g, \delta}$ and the Lipschitz continuity  and boundedness  of $G^\delta$,
 and  applying the Gronwall inequality imply 
\begin{equation}\label{estim_contr_1}
\begin{aligned}
&\|\F_{\rg, \delta}^1 - \F_{\rg, \delta}^2\|^2_{L^\infty((0,  T)\times \Omega^\delta)} \leq
  T C \big\|G^\delta(x, \nabla \bu^\delta_1, (\tilde \F^1_{\rg, \delta})^{-1}) - G^\delta(x, \nabla \bu^\delta_2, (\tilde \F^2_{\rg, \delta})^{-1})\big\|^2_{L^\infty((0,T)\times \Omega^\delta)}
\\
& \qquad \quad \leq
T \big[ C_\delta   \|{\rm{sym}}(\nabla \bu^\delta_1 (\tilde \F_{\rm g, \delta}^1)^{-1}) - {\rm{sym}}(\nabla \bu^\delta_2 (\tilde \F^2_{\rm g, \delta})^{-1})\|^2_{L^\infty(0,T; L^2(\Omega^\delta))} + C \|\tilde \F_{\rg, \delta}^1 - \tilde \F_{\rg, \delta}^2 \|^2_{L^\infty((0,T)\times\Omega^\delta)}\big],
\end{aligned}
\end{equation}
for $\tilde \F^1_{\rg, \delta}, \tilde \F^2_{\rg, \delta} \in \mathcal W$, where constant  $C$  depends on $\|\tilde \F^j_{\rg, \delta}\|_{L^\infty((0,T)\times \Omega^\delta)}$, for $j=1,2$, which is bounded by $\hat\kappa$, and
$C_\delta$ depends on $\|\tilde \F^j_{\rg, \delta}\|_{L^\infty((0,T)\times \Omega^\delta)}$ and  $\delta$.  
In the derivation of  \eqref{estim_contr_1} we  used the following estimate 
$$
\begin{aligned}
& \big\|G^\delta(x, \nabla \bu^\delta_1, (\tilde \F^1_{\rg, \delta})^{-1}) - G^\delta(x, \nabla \bu^\delta_2, (\tilde \F^2_{\rg, \delta})^{-1})\big\|^2_{ L^\infty(\Omega^\delta)}
\\
& \qquad \qquad \leq 
C \Big\|\ddashinttt_{\delta( [x/\delta]_{Y}+  Y_w)\cap \Omega} \Big| {\boldsymbol\ve}^{\rm el}(\nabla \bu^\delta_1, (\tilde\F_{\rg, \delta}^1)^{-1}) -  {\boldsymbol\ve}^{\rm el}(\nabla \bu^\delta_2, (\tilde \F_{\rg, \delta}^2)^{-1})  \Big|  d\tilde x \Big\|^2_{ L^\infty(\Omega^\delta)}
\\
& \qquad \qquad  \leq C_1 \delta^{-d}\big\| {\rm{sym}}\big(\nabla \bu^\delta_1 (\tilde \F_{\rm g, \delta}^1\big)^{-1})  -{\rm{sym}} \big(\nabla \bu^\delta_2 (\tilde \F^2_{\rm g, \delta})^{-1}\big)  \big\|^2_{ L^2(\Omega^\delta)} + C_2 \big\|\tilde \F_{\rg, \delta}^1 - \tilde \F_{\rg, \delta}^2 \big\|^2_{ L^\infty(\Omega^\delta)},
\end{aligned}
$$
for a.a.~$t \in (0,T)$. 
From  \eqref{micro_model_ref_2}, using the uniform ellipticity of $\bbE$ and estimate \eqref{estim-u-delta_1}, we obtain
\begin{equation} \label{estim_contr_2}
\begin{aligned}
\big\|{\rm{sym}}(\nabla \bu^\delta_1 (\tilde \F_{\rm g, \delta}^1)^{-1})  -{\rm{sym}} (\nabla \bu^\delta_2 (\tilde \F^2_{\rm g, \delta})^{-1})\big\|^2_{L^\infty(0,T; L^2(\Omega^\delta))}
& \leq  C \big\|\tilde \F_{\rg, \delta}^1 - \tilde  \F_{\rg, \delta}^2\big\|^2_{L^\infty((0,T)\times \Omega^\delta)}.
%\\ & \leq C T\|\tilde \F_{\rg, \delta}^1 - \tilde \F_{\rg, \delta}^2\|^2_{L^\infty(0, T; L^\infty(\Omega^\delta))} .
\end{aligned}
\end{equation}
In the derivation of \eqref{estim_contr_2} we used the  following  estimate 
$$
\begin{aligned} 
\big\|\nabla(\bu_1^\delta - \bu_2^\delta)\big[(\tilde \F_{\rm g, \delta}^1)^{-1} - (\tilde \F_{\rm g, \delta}^2)^{-1}\big]\big\|_{L^2(\Omega^\delta)}
\leq  C_{\tilde \varsigma}\big\|(\tilde \F_{\rm g, \delta}^1)^{-1}\big\|^2_{L^\infty(\Omega^\delta)} \big\|(\tilde \F_{\rm g, \delta}^2)^{-1}\big\|^2_{L^\infty(\Omega^\delta)}\big\|\tilde \F_{\rm g, \delta}^1 - \tilde \F_{\rm g, \delta}^2\big\|^2_{L^\infty(\Omega^\delta)}\\  +  \tilde \varsigma \big\|\nabla(\bu_1^\delta - \bu_2^\delta)\big\|^2_{L^2(\Omega^\delta)}
\leq \varsigma\big\|{\rm sym}\big(\nabla(\bu_1^\delta - \bu_2^\delta)(\tilde\F_{\rm g, \delta}^1)^{-1}\big)\big\|^2_{L^2(\Omega^\delta)} + C_1\big\|\tilde \F_{\rm g, \delta}^1 - \tilde \F_{\rm g, \delta}^2
\big\|^2_{L^\infty(\Omega^\delta)}
\\
\leq \varsigma\big\|{\rm sym}\big(\nabla \bu_1^\delta (\tilde\F_{\rm g, \delta}^1)^{-1}\big) - {\rm sym}\big(\nabla \bu_2^\delta(\tilde\F_{\rm g, \delta}^2)^{-1}\big)\big\|^2_{L^2(\Omega^\delta)}
 + C\big(1+\|\nabla \bu_2^\delta\|_{L^2(\Omega^\delta)}^2\big) \big\|\tilde \F_{\rm g, \delta}^1 - \tilde \F_{\rm g, \delta}^2 \big\|^2_{L^\infty(\Omega^\delta)},
\end{aligned}
$$
for any fixed $\varsigma>0$. Here we applied  \eqref{estim:Korn_gen}  with  $(\tilde\F_{\rm g, \delta}^1)^{-1}$ and $\bu_1^\delta(t) - \bu_2^\delta(t) \in V_\delta$.  

Combining estimates \eqref{estim_contr_1} and  \eqref{estim_contr_2} and considering $T$ sufficiently small we obtain that $\mathcal K$ is a contraction.  Then applying the Banach fixed point theorem yields existence of the unique solution for  \eqref{growth_micro_1} and \eqref{micro_model_ref_2}. Since the choice  of $T$ depends only on the parameters in the system and on $\delta$,  and does not depend on the solution, we can iterate over time intervals to obtain the global existence and uniqueness result for the microscopic model \eqref{growth_micro_1} and \eqref{micro_model_ref_2} for any fixed $\delta >0$.
\end{proof}

Next we prove convergence results for sequences $\{\bu^\delta\}$ and $\{\F_{\rg, \delta}\}$ as $\delta \to 0$.
\begin{lemma} \label{lem:conv_stong}
Under Assumptions~\ref{assumption}, for sequence of  solutions   $\{\bu^\delta\}$ and $\{\F_{\rg, \delta}\}$  of microscopic model  \eqref{growth_micro_1} and~\eqref{micro_model_ref_2}, up to a subsequence,  we  have the following convergence results
\begin{equation}\label{convergence_1}
 \begin{aligned} 
 & \bu^\delta \rightharpoonup \bu && \text{ weakly-$\ast$   in } \;  && L^\infty(0,T ; V), \\
  & \nabla \bu^\delta \rightharpoonup \nabla \bu + \nabla_y \bu_1 && \text{ two-scale},  \;&&  \bu_1 \in L^2((0,T)\times \Omega; H^1_{\rm per}(Y)/ \mathbb R), \\
  & \F_{\rg, \delta} \rightharpoonup \F_\rg  && \text{ weakly-$\ast$ in } \; && L^\infty((0, T) \times \Omega),\\
  & \F_{\rg, \delta} \rightharpoonup \F_\rg  && \text{ weakly-$\ast$ in } \; && W^{1,\infty}(0, T; L^q(\Omega)), \quad \text{ for any }  q \in (1, \infty),
   \end{aligned} 
  \end{equation} 
  where we identify $\bu^\delta$ with its extension from $\Omega^\delta$  into $\Omega$ and for $\F_{\rg, \delta}$ consider a trivial (constant) extension from $\delta(Y_w + \xi)$ into $\delta(Y+\xi)$ for $\xi \in \Xi_\delta$. The space $V$ is defined in the same way as $V_\delta$, with $\Omega^\delta$ replaced by $\Omega$.
  
  Assuming additionally  that 
  \begin{equation}\label{bound_bu}
   \|\mathcal T^\delta(\nabla\bu^\delta)\|_{L^\infty((0,T)\times \Omega; L^2(Y_w))} \leq C,
  \end{equation}
  with a constant $C$ independent of $\delta$, we have
  \begin{equation}\label{convergence_2}
 \begin{aligned} 
    \F_{\rg, \delta} & \to \F_\rg  && \text{ strongly in}  && L^2((0,T)\times \Omega), \\
    \int_{Y_w} \mathcal T^\delta ({\rm sym}(\nabla \bu^\delta \F_{\rg, \delta}^{-1})) \, dy
    & \to  \int_{Y_w}  {\rm sym}((\nabla \bu + \nabla_y \bu_1) \F_{\rg}^{-1}) \, dy && \text{ strongly in}  && L^2((0,T)\times \Omega),
    \\
   \int_{Y_w} \mathcal T^\delta \big(\mathbb E^\delta(x){\rm sym}(\nabla \bu^\delta \F_{\rg, \delta}^{-1})\big) \, dy & \to  \int_{Y_w} \mathbb E(\cdot,y) \, {\rm sym}((\nabla \bu + \nabla_y \bu_1) \F_{\rg}^{-1}) \, dy && \text{ strongly in}  && L^2((0,T)\times \Omega),
  \end{aligned} 
  \end{equation} 
  where $\mathcal T^\delta$ is the periodic unfolding operator defined below. 
\end{lemma}
\begin{proof}
The first four convergences follow directly from  estimates \eqref{estim_apriori_1} and compactness results for the weak-$\ast$  and two-scale convergences,    see~e.g.~\cite{Allaire92, Nguetseng89} or Appendix for the definition and properties of the two-scale convergence.

To show the strong convergence of $\F_{\rg, \delta}$ we shall use the periodic unfolding operator $\mathcal T^\delta: L^p((0,T)\times \Omega^\delta) \to L^p((0,T)\times \Omega\times Y_w)$, with $1\leq p \leq \infty$, defined as 
\begin{equation}\label{unfolding_Yw}
\mathcal T^\delta (v)(t,x,y) = \begin{cases} v(t,\delta[x/\delta]_Y+ \delta y) \quad &  \text{ for a.e. }  \; (x,y) \in (\Omega\setminus \Lambda_\delta) \times Y_w,\, t \in (0,T),\\
0 & \text{ for a.e. } \; (x,y) \in \Lambda_\delta\times Y_w, \; t \in (0,T),
\end{cases} 
\end{equation}
where  $[x/\delta]_Y$ denotes the unique integer combination, such that  $x/\delta - [x/\delta]_Y$ belongs to $Y_w$ for $x \in \Omega^\delta \setminus \Lambda_\delta$,  see  e.g.~\cite{CDDGZ, CDG_book} and Appendix for more details. 

Applying the periodic unfolding operator $\mathcal T^\delta $ to  \eqref{growth_micro_1} yields 
\begin{equation} \label{unfold_growth}
\partial_t \mathcal T^\delta (\F_{\rg, \delta})= 
 G^\delta(x,\mathcal T^\delta(\nabla \bu^\delta), \mathcal T^\delta ( \F_{\rg, \delta}^{-1}) ) \, \mathcal T^\delta(\F_{\rg, \delta}) \quad \text{ in } (0,T)\times \Omega \times Y_w, 
\end{equation}
where 
$$
\tilde G^\delta(x,\mathcal T^\delta(\nabla \bu^\delta), \mathcal T^\delta ( \F_{\rg, \delta}^{-1}))  =  \eta_\sigma \Big[ \ddashinttt_{Y_w} \boldsymbol\sigma(x,y, \mathcal T^\delta(\nabla \bu^\delta), \mathcal T^\delta ( \F_{\rg, \delta}^{-1} )) \, dy - \boldsymbol{\tau}_\sigma \Big]_{+}, 
$$
or 
$$
\tilde G^\delta(x,\mathcal T^\delta(\nabla \bu^\delta), \mathcal T^\delta ( \F_{\rg, \delta}^{-1}))  = \eta_\varepsilon \Big[ \ddashinttt_{Y_w} \boldsymbol{\varepsilon}^{\rm el}(\mathcal T^\delta(\nabla \bu^\delta), \mathcal T^\delta ( \F_{\rg, \delta}^{-1}) ) \, dy - \boldsymbol{\tau}_\varepsilon \Big]_{+},
$$
and $G^\delta$ is defined in terms of $\tilde G^\delta$ as in \eqref{def_G}.  
Notice that $G^\delta$ is independent of $y$ and hence $\mathcal T^\delta (\F_{\rg, \delta})$ depends on $t\in (0,T)$ and $x \in \Omega$ and is constant in $y\in Y_w$.

We shall show that $\{ \mathcal T^\delta ( \F_{\rg, \delta})\} $ converges strongly in $L^2((0,T)\times \Omega \times Y_w)$ by applying the Fr\'echet-Kolmogorov compactness theorem, \cite{Necas}, and its modification proposed by Simon, \cite{Simon_1986}. 
The uniform boundedness of  $G^\delta$ implies
\begin{equation}\label{estim_bound_FFt}
\|\mathcal T^\delta (\F_{\rg, \delta}) \|_{L^\infty((0,T)\times \Omega\times Y_w)} + \|\partial_t \mathcal T^\delta (\F_{\rg, \delta}) \|_{L^\infty((0,T)\times \Omega\times Y_w)} \leq C,
\end{equation}
where the constant $C$ is independent of $\delta$. 

 Denote  $\Omega_h = \{ x \in \Omega :  {\rm dist} (\partial \Omega, x) > h\}$ for  $h>0$.
Considering \eqref{growth_micro_1} for $x+\hat h$ and $x$, with    $\hat h=\delta \xi$, for $\xi \in \Xi$, and $|\hat h| \leq h$, applying the  unfolding operator $\mathcal T^\delta$,  multiplying the resulting equations by  $\mathcal T^\delta\big(\F_{\rg, \delta}(x+\hat h)\big) - \mathcal T^\delta\big(\F_{\rg, \delta}(x)\big)$,  integrating over $(0,t)$, taking supremum over $Y_w$, and then  integrating over $\Omega_{3h}$, we obtain    \begin{equation}\label{strong_estim_1}
\begin{aligned}
&\big\| \mathcal T^\delta (\F_{\rg, \delta}(\cdot +\hat h)) -
 \mathcal T^\delta (\F_{\rg, \delta})\big\|^2_{L^2(\Omega_{3h}; L^\infty(Y_w))}
 \leq C \int_0^t \Big(\big\| \mathcal T^\delta (\F_{\rg, \delta}(\cdot +\hat h)) - 
 \mathcal T^\delta (\F_{\rg, \delta})\big\|^2_{L^2(\Omega_{3h}; L^\infty(Y_w))}  \\
 & \; +  \Big\|\ddashinttt_{Y_w} \Big[\boldsymbol \sigma\big(\cdot+\hat h, y, \mathcal T^{\delta}(\nabla \bu^{\delta}(\cdot+\hat h)), \mathcal T^{\delta}( \F^{-1}_{\rg, \delta}(\cdot+\hat h))\big) - \boldsymbol \sigma\big(\cdot, y,\mathcal T^{\delta}(\nabla \bu^{\delta}), \mathcal T^{\delta}( \F^{-1}_{\rg, \delta})\big) \Big] dy \Big\|^2_{L^2(\Omega_{3h})} \Big)ds,
\end{aligned}
\end{equation}
or with $\boldsymbol{\varepsilon}^{\rm el}$ instead of $\boldsymbol {\sigma}$ if we consider the second case for the growth rate. Notice that for the ease of notations in the formulas here and below we often do not write explicitly the dependence of $\F_{\rg, \delta}$,  $\bu^\delta$, $P_1$ and $P_2$ on $t \in [0,T]$.

To estimate the second term on the right-hand site of \eqref{strong_estim_1},  we consider equations \eqref{micro_model_ref_2} for $x$ and $x+\hat h$ and take $(\bu^{\delta}(t,x+\hat h) -  \bu^{\delta}(t,x))\rho_h^2(x)$ as a test function, where $\rho_h \in C^1_0(\Omega)$ with  $\rho_h (x)= 1$ in $\Omega_{3h}$ and $\rho_h(x) =0$ in $\Omega \setminus \Omega_{2h}$.
Then applying the periodic unfolding operator $\mathcal T^\delta$ and using   \eqref{bound_bu}, together with the uniform boundedness of $\F_{\rg, \delta}$ and assumptions on $\mathbb E$, $P_1$, and $P_2$, yield
\begin{equation}\label{estim_string_bu_11}
\begin{aligned} 
&\Big\langle  \mathcal T^{\delta}(J_\rg^\delta \F_{\rg, \delta}^{-1}\mathbb E^\delta(x+\hat h)) {\rm sym}(\mathcal T^{\delta}(\nabla \bu^\delta \F_{\rg, \delta}^{-1}(x+\hat h)))
 - \mathcal T^{\delta}(J_\rg^\delta  \F_{\rg, \delta}^{-1} \mathbb E^\delta(x)) {\rm sym}(\mathcal T^{\delta}(\nabla \bu^\delta \F_{\rg, \delta}^{-1}(x))) \\
& \quad + \mathcal T^{\delta} (J_\rg^\delta \F_{\rg, \delta}^{-1} P_1(x+\hat h)) - \mathcal T^{\delta} (J_\rg^\delta \F_{\rg, \delta}^{-1}P_1(x)), \big[\mathcal T^{\delta} (\nabla \bu^{\delta}(x+\hat h) ) -  \mathcal T^\delta(\nabla \bu^{\delta}(x)) \big]\mathcal T^{\delta} (\rho_h^2) \Big\rangle_{\Omega_{2h}\times Y_w}\\
 & \quad  + \Big \langle \mathcal T^{\delta}(J_{\rg}^\delta P_2^\delta\F_{\rg, \delta}^{-T}   N(x+\hat h))- \mathcal T^{\delta}(J_{\rg}^\delta P_2^\delta  \F_{\rg, \delta}^{-T}N(x)),  \big[\mathcal T^{\delta}(\bu^\delta(x+\hat h)) - \mathcal T^\delta(\bu^\delta(x))\big] \mathcal T^{\delta} (\rho_h^2)\Big \rangle_{\Omega_{2h} \times \Gamma}   \\
 & = 
- \Big
\langle  \mathcal T^{\delta}(J_\rg^\delta \F_{\rg, \delta}^{-T}\nabla P_1(x+\hat h))- \mathcal T^{\delta}(J_\rg^\delta \F_{\rg, \delta}^{-T} \nabla P_1(x)), \big[\mathcal T^{\delta}(\bu^\delta(x+\hat h)) - \mathcal T^\delta(\bu^\delta(x))\big] \mathcal T^{\delta} (\rho_h^2) \Big\rangle_{\Omega_{2h}\times Y_w} \\
&\quad + \Big
\langle \mathcal T^{\delta}( J_\rg^\delta \F_{\rg, \delta}^{-1} \mathbb E^\delta (x+\hat h)) \big[\bI -{\rm sym} (\mathcal T^{\delta}(\F_{\rg, \delta}^{-1}(x+\hat h)))\big]
\\ &\qquad \quad  -
\mathcal T^{\delta}( J_\rg^\delta \F_{\rg, \delta}^{-1} \mathbb E^\delta(x) )\big[\bI -{\rm sym} (\mathcal T^{\delta}(\F_{\rg, \delta}^{-1}(x)))\big], \big[ \mathcal  T^{\delta}( \nabla \bu^\delta(x+\hat h) )- \mathcal  T^{\delta}(\nabla\bu^\delta(x))\big] \mathcal T^{\delta} (\rho_h^2) \Big\rangle_{\Omega_{2h}\times Y_w}
\\
& \quad + \kappa(h),
\end{aligned} 
\end{equation}
where $\kappa(h) \to 0$ as $h\to 0$. 
Here we used  that $|\hat h |\leq h$,  that $\rho_h(x)=0$ and $\rho_h(x\pm \hat h) = 0$ for $x \in \Lambda_\delta$,   and the following  estimate
\begin{equation}\label{estim_rhoh_deriv}
\begin{aligned} 
&\big\|
\mathcal T^{\delta} \big( \bu^{\delta}(\cdot+\hat h) -  \bu^{\delta}\big) \mathcal T^{\delta} (\rho_h \nabla \rho_h) \big\|_{L^2(\Omega_{2h}\times Y_w)}\\
& \qquad  \leq C \frac {|\hat h|} {h} \big\|  \mathcal T^{\delta}(\nabla\bu^\delta)\big\|_{L^2((\Omega_{2h} \setminus \Omega_{3h})\times Y_w)}  \leq C |\Omega_{2h} \setminus \Omega_{3h}|^{1/ 2} \| \mathcal T^{\delta}(\nabla\bu^\delta)\|_{L^\infty((0,T)\times \Omega; L^2(Y_w))}  \leq \kappa(h),
\end{aligned} 
\end{equation}
for a.a.~$t \in (0,T)$.
Then using \eqref{bound_bu}, together with the uniform boundedness of $\F_{\rg, \delta}$,  the  assumptions on $\mathbb E$, $P_1$, and $P_2$,  and  that  integrals over $(\Omega_{2h} \setminus \Omega_{3h}) \times Y_w$ can be estimated by $|\Omega_{2h} \setminus \Omega_{3h}|$, from \eqref{estim_string_bu_11}  we obtain
\begin{equation} \label{estim_diff_elast}
\begin{aligned}
 \big\|{\rm sym}\big(\mathcal T^{\delta} (\nabla \bu^\delta \F^{-1}_{\rg, \delta}(\cdot+\hat h)) - 
\mathcal T^{\delta} (\nabla \bu^\delta \F^{-1}_{\rg, \delta})\big)\big\|_{L^2(\Omega_{3h}\times Y_w)}^2 & \leq \kappa(h) \\
  +  & C \big\|  \mathcal T^{\delta}(\F_{\rg, \delta} (\cdot+\hat h)) -  \mathcal T^{\delta}(\F_{\rg, \delta}) \big \|^2_{L^2(\Omega_{3h}; L^\infty(Y_w))},
\end{aligned}
\end{equation}
for a.a.~$t \in (0,T)$, where $\kappa(h) \to 0$ as $h\to 0$ and $|\hat h|\leq h$.  Similar as in the proof of Theorem~\ref{th:exist}, 
in the derivation of  \eqref{estim_diff_elast} we used  the following estimate 
$$
\begin{aligned}
& \big\|\mathcal T^{\delta} (\nabla \bu^{\delta}(\cdot+\hat h) -  \nabla \bu^{\delta})\, \mathcal T^{\delta}(\rho_h)\big[\mathcal T^{\delta} (\F^{-1}_{\rg, \delta}(\cdot +\hat h)) -  \mathcal T^{\delta}(\F^{-1}_{\rg, \delta}) \big]\big\|_{L^1(\Omega_{2h}; L^2(Y_w))} \\
& \quad \leq  \tilde\varsigma 
\|\nabla (\bu^{\delta}(\cdot +\hat h) -  \bu^\delta) \rho_h \|^2_{L^2(\Omega^\delta) }
 + C_{\tilde \varsigma} \big \|\mathcal T^{\delta} (\F^{-1}_{\rg, \delta}(\cdot +\hat h) ) -  \mathcal T^{\delta}(\F^{-1}_{\rg, \delta})\big \|^2_{L^2(\Omega_{2h}; L^\infty(Y_w))}
\\
& \quad \leq \varsigma \big\|{\rm {sym}}\big(\nabla (\bu^{\delta}(\cdot +\hat h) - \bu^\delta) \F^{-1}_{\rg, \delta}\big) \rho_h\big\|^2_{L^2(\Omega^\delta)}  + C_{\tilde \varsigma} \big \|\mathcal T^{\delta} (\F_{\rg, \delta}(\cdot +\hat h)) -  \mathcal T^{\delta}  (\F_{\rg, \delta})\big \|^2_{L^2(\Omega_{2h}; L^\infty(Y_w))} + \kappa(h)\\
&  \quad \leq \varsigma \big\|{\rm {sym}}\big(\nabla \bu^{\delta}(\cdot +\hat h) \F^{-1}_{\rg, \delta}(\cdot + \hat h)- \nabla \bu^\delta \F^{-1}_{\rg, \delta}\big) \rho_h\big\|^2_{L^2(\Omega^\delta)}  + \kappa(h) 
\\
& \quad \qquad + \varsigma \big\|{\rm {sym}}\big(\nabla \bu^{\delta}(\cdot +\hat h)\big[\F^{-1}_{\rg, \delta} (\cdot +\hat h)- \F^{-1}_{\rg, \delta}\big]\big)  \rho_h\big\|^2_{L^2(\Omega^\delta)}
 +  C \big \|\mathcal T^{\delta} (\F_{\rg, \delta}(\cdot +\hat h) ) -  \mathcal T^{\delta}  (\F_{\rg, \delta})\big \|^2_{L^2(\Omega_{2h}; L^\infty(Y_w))} \\
& \quad  \leq \varsigma  \big\|{\rm {sym}}\big(\mathcal T^\delta(\nabla \bu^{\delta} \F^{-1}_{\rg, \delta}(\cdot +\hat h))- \mathcal T^\delta(\nabla \bu^\delta \F^{-1}_{\rg, \delta})\big) \big\|^2_{L^2(\Omega_{2h}\times Y_w)}
 \\ & \quad \qquad
 + C  \big \|\mathcal T^\delta(\F_{\rg, \delta} (\cdot +\hat h))- \mathcal T^\delta(\F_{\rg, \delta})  \big\|^2_{L^2(\Omega_{2h}; L^\infty(Y_w))} + \kappa(h),
\end{aligned}
$$
for any fixed $\varsigma >0$, a.a.~$t \in (0,T)$,  and $\kappa(h) \to 0$ as $h \to 0$.
In the derivation of the last estimate  we used an estimate similar to \eqref{estim_rhoh_deriv},  the properties of the unfolding operator and estimate \eqref{estim:Korn_gen} applied  to $(\bu^{\delta}(\cdot +\hat h) -  \bu^\delta) \rho_h$ and $\F^{-1}_{\rg, \delta}$. Notice that  $(\bu^{\delta}(\cdot +\hat h) -  \bu^\delta) \rho_h =0 $ on $\partial \Omega$.

Using \eqref{estim_diff_elast} in \eqref{strong_estim_1}, together with the regularity of $\mathbb E$ with respect to the first variable, and applying the Gr\"onwal inequality yields
\begin{equation}\label{diff_growth_2}
\big\|  \mathcal T^{\delta}(\F_{\rg, \delta} (\cdot, \cdot+\hat h)) -  \mathcal T^{\delta}(\F_{\rg, \delta}) \big\|^2_{L^\infty(0,T; L^2(\Omega_{3h}; L^\infty(Y_w)))} \leq \kappa(h),
\end{equation}
and, using \eqref{estim_diff_elast},  also 
\begin{equation} \label{estim_diff_elast_2}
\begin{aligned}
\big\|{\rm sym}\big(\mathcal T^{\delta} (\nabla \bu^\delta \F^{-1}_{\rg, \delta}(\cdot, \cdot+\hat h)) - 
\mathcal T^{\delta} (\nabla \bu^\delta \F^{-1}_{\rg, \delta})\big) \big\|_{L^\infty(0,T; L^2(\Omega_{3h}\times Y_w))}^2 \leq \kappa(h),
\end{aligned}
\end{equation}
where $\kappa(h) \to 0$ as  $h \to 0$ and $|\hat h| \leq h$. 
In a similar way, using  the uniform boundedness of $\partial_t \F_{\rg, \delta}$, and hence also of $\partial_t \mathcal T^{\delta}(\F_{\rg, \delta})$,   and uniform continuity, with respect to the time variable, of $P_1$, $\nabla P_1$, $P_2$, and ${\bf f}$,  we  obtain
\begin{equation} \label{estim_diff_elast_3}
\begin{aligned}
\big\|{\rm sym}\big(\mathcal T^{\delta} (\nabla \bu^\delta \F^{-1}_{\rg, \delta})(\cdot+\tilde\tau, \cdot, \cdot ) -
\mathcal T^{\delta} (\nabla \bu^\delta \F^{-1}_{\rg, \delta})\big)\big\|_{L^2((0,T-\tilde \tau)\times\Omega\times Y_w)}^2 \leq \tilde \kappa(\tilde \tau), \quad \text{ with } \; \tilde \kappa(\tilde \tau) \to 0 \; \text{ as } \; \tilde \tau \to 0.
\end{aligned}
\end{equation}
From the definition of the periodic unfolding operator, 
for any $\tilde h \in \mathbb R^d$, with $|\tilde h| < h$, we have 
\begin{equation*} 
\begin{aligned} 
& \big\| \mathcal T^\delta (\F_{\rg, \delta})(\cdot, \cdot +\tilde  h, \cdot) -
 \mathcal T^\delta (\F_{\rg, \delta}) \big\|^2_{L^2(\Omega_{3h, T}\times Y_w)}
  \leq |Y| \sum_{k \in \{0, 1\}^d}
\big \| \F_{\rg, \delta}(\cdot, \cdot +\delta (k_b +[\tilde h/\delta]_Y)) -
 \F_{\rg, \delta}\big\|^2_{L^2(\Omega^\delta_{3h, T})}
 \\
 & \leq \sum_{k \in \{0, 1\}^d}
\big \| \mathcal T^\delta(\F_{\rg, \delta}(\cdot, \cdot +\hat h)) -
 \mathcal T^\delta(\F_{\rg, \delta})\big\|^2_{L^2(\Omega_{3h, T}\times Y_w)}
 \leq \sum_{k \in \{0, 1\}^d}
 \big\| \mathcal T^\delta(\F_{\rg, \delta}(\cdot, \cdot +\hat h)) -
 \mathcal T^\delta(\F_{\rg, \delta})\big\|^2_{L^2(\Omega_{3h, T}; L^\infty (Y_w))},
\end{aligned} 
\end{equation*} 
where  $\hat h = \delta (k_b +[\tilde h/\delta]_Y)$, with $k_b= \sum_{j=1}^d k_j b_j$  and $|\hat h|\leq h$  for sufficiently small~$\delta$, and $\Omega_{3h, T} = (0,T)\times \Omega_{3h}$.
Then for any $h>0$ there exists $\delta_0$, such that $|\hat h| = |\delta (k_b +[\tilde h/\delta]_Y)| \leq h$ for all $\delta \leq \delta_0$.
Considering \eqref{diff_growth_2} for such $\hat h$ yields
\begin{equation} \label{estim_Kolg}
\big\| \mathcal T^\delta (\F_{\rg, \delta})(\cdot, \cdot +\tilde  h, \cdot) -
 \mathcal T^\delta (\F_{\rg, \delta}) \big\|^2_{L^2((0,T)\times\Omega_{3h}\times Y_w)} \leq  \kappa(h),
\end{equation} 
where $\kappa(h) \to 0$ as $ h \to 0$, and   hence as $|\tilde h| \to 0$, 
for all $\delta\leq \delta_0$. For $\delta > \delta_0$ we have a finite number of members  of the sequence  and for each of them the continuity of the $L^2$-norm of  $L^2$-functions, see e.g.~\cite{Alt_FA},   ensures \eqref{estim_Kolg} for an appropriate $\tilde h$, with $|\tilde h| < h$. Then  from the finite number of such $\tilde h$ considering the smallest one 
 implies  the property \eqref{estim_Kolg} for all $\delta >0$. 
Similarly we obtain 
\begin{equation} \label{estim_diff_elast_4}
\begin{aligned}
\big\|{\rm sym}\big(\mathcal T^{\delta} (\nabla \bu^\delta \F^{-1}_{\rg, \delta})(\cdot, \cdot+\tilde h,\cdot) -
\mathcal T^{\delta} (\nabla \bu^\delta \F^{-1}_{\rg, \delta})\big)\big\|_{L^2((0,T)\times\Omega_{3h}\times Y_w)}^2 \leq \kappa(h),
\end{aligned}
\end{equation}
for  $\tilde h \in \mathbb R^d$, with  $|\tilde h|<h$, and all $\delta>0$, where $\kappa(h) \to 0$ as $h \to 0$.
Thus  using the uniform boundedness of $\F_{\rg, \delta}$, assumption \eqref{bound_bu}, and the fact that $\mathcal T^\delta(\F_{\rg, \delta})$ is constant in $y\in Y_w$, and applying  the   Fr\'echet-Kolmogorov and Simon compactness theorems, see~\cite{Simon_1986}, we obtain  the strong convergence of $\mathcal T^\delta(\F_{\rg, \delta})$  to $\tilde \F_\rg$ in $L^2((0,T)\times \Omega \times Y_w)$  and of 
$\int_{Y_w} {\rm sym}\big(\mathcal T^{\delta} (\nabla \bu^\delta \F^{-1}_{\rg, \delta}) ) dy$ in $L^2((0,T)\times \Omega)$. 
The properties of $\mathcal T^\delta$ and  that $\tilde \F_\rg$ and  $\mathcal T^{\delta} (\F_{\rg, \delta})$ are constant in~$y\in Y_w$ imply
$$
\begin{aligned} 
& \int_{\Omega_T} \F_{\rg} \psi \, dx dt = 
\lim\limits_{\delta \to 0 } \int_{\Omega_T} \F_{\rg, \delta} \psi \, dx dt 
= \frac 1{|Y|} \lim\limits_{\delta \to 0 } \int_{\Omega_T} \int_{Y} \mathcal T^\delta(\F_{\rg, \delta}) \mathcal T^\delta( \psi) \, dy dx dt \\
& = 
\frac 1{|Y_w|} \lim\limits_{\delta \to 0 } \int_{\Omega_T} \int_{Y_w} \mathcal T^\delta(\F_{\rg, \delta}) \mathcal T^\delta( \psi) \, dy dx dt
= \frac 1{|Y_w|}  \int_{\Omega_T} \int_{Y_w} \tilde \F_{\rg} \mathcal \psi \, dy dx dt = \int_{\Omega_T}  \tilde \F_{\rg} \mathcal \psi \,  dx dt,
\end{aligned} 
$$
for all $\psi \in L^2((0,T)\times \Omega)$. Hence $\tilde \F_\rg = \F_\rg$ in $(0,T)\times \Omega$. 
Similarly we obtain  
$$
\begin{aligned} 
&\big\|\F_{\rg, \delta} - \F_{\rg}\big\|^2_{L^2((0,T)\times \Omega)}\leq
|Y|^{-1}\big\| \mathcal T^\delta(\F_{\rg, \delta})-
\mathcal T^\delta (\F_\rg) \big\|^2_{L^2((0,T)\times \Omega \times Y)} + C |\Lambda_\delta|\\
&\leq  |Y_w|^{-1}
\big\| \mathcal T^\delta(\F_{\rg, \delta})-
\F_\rg \big\|^2_{L^2((0,T)\times \Omega \times Y_w)}
+ 
|Y|^{-1} \big\| \mathcal T^\delta(\F_{\rg})-
\F_\rg \big\|^2_{L^2((0,T)\times \Omega \times Y)} + C|\Lambda_\delta|
\to 0, \; \text{ as }   \delta \to 0, 
\end{aligned} 
$$
which ensures  the strong convergence of $\F_{\rg, \delta}$ to $\F_\rg$ in $L^2((0,T)\times \Omega)$.
Strong convergence of $\F_{\rg, \delta}$, together with the uniform boundedness of $\F_{\rg, \delta}$ and ${\rm det}(\F_{\rg, \delta})\geq 1$,  implies also strong convergence of  $\F_{\rg, \delta}^{-1}$ in $L^2((0,T)\times \Omega)$.  Then using the equivalence between two-scale convergence of a sequence and weak convergence of the corresponding unfolded sequence we obtain  $\int_{Y_w} {\rm sym}\big(\mathcal T^{\delta} (\nabla \bu^\delta \F^{-1}_{\rg, \delta}) ) dy \to \int_{Y_w} {\rm sym} ((\nabla \bu + \nabla_y \bu_1) \F_{\rg}^{-1}) dy$ strongly in $L^2((0,T)\times \Omega)$.

Since the elasticity tensor $\mathbb E$ is independent of the time variable,  estimate \eqref{estim_diff_elast_3} implies
\begin{equation} \label{estim_diff_elast_stress_1}
\begin{aligned}
\big\|\mathcal T^{\delta}\big(\mathbb E^\delta {\rm sym}(\nabla \bu^\delta \F^{-1}_{\rg, \delta})\big)(\cdot+\tilde \tau, \cdot, \cdot ) -
\mathcal T^{\delta} \big(\mathbb E^\delta {\rm sym}(\nabla \bu^\delta \F^{-1}_{\rg, \delta})\big)\big\|_{L^2((0,T-\tilde \tau)\times\Omega\times Y_w)}^2 \leq \tilde \kappa(\tilde \tau),
\end{aligned}
\end{equation}
 with $\tilde \kappa(\tilde \tau) \to 0$ as $\tilde \tau \to 0$. From \eqref{estim_diff_elast_4}, together with the regularity of $\mathbb E$ with respect to the first variable and $Y$-periodicity with respect to the second variable, we obtain
$$
\begin{aligned}
&\big \|\mathcal T^{\delta}(\mathbb E^\delta)(\cdot+\tilde h,\cdot){\rm sym}\big(\mathcal T^{\delta} (\nabla \bu^\delta \F^{-1}_{\rg, \delta})\big)(\cdot, \cdot+\tilde h,\cdot) -
\mathcal T^{\delta}(\mathbb E^\delta) {\rm sym}\big(\mathcal T^{\delta} (\nabla \bu^\delta \F^{-1}_{\rg, \delta})\big)\big\|_{L^2((0,T)\times\Omega_{3h}\times Y_w)}^2 \\
& \leq \big \|\mathbb E\big\|_{L^\infty(\Omega\times Y_w)}^2\, \kappa(h)
+ \sum\limits_{ k \in \{ 0, 1\}^d}\big\|\mathbb E(\cdot+\hat h, \cdot) - \mathbb E\big\|_{L^\infty(\Omega_{3h}\times Y_w)}^2 \big\|{\rm sym}\big( \nabla \bu^\delta \F^{-1}_{\rg, \delta}\big)\big\|_{L^2((0,T)\times\Omega^\delta)}^2 \leq \kappa_1(h),
\end{aligned}
$$
where $\hat h = \delta (k_b + [\tilde h/\delta]_Y)$ for  $\tilde h \in \mathbb R^d$ with $|\tilde h |< h$ and $\kappa_1(h)\to 0$ as $h \to 0$. Then using the two-scale convergence of $\nabla \bu^\delta$ to $\nabla \bu + \nabla_y \bu_1$, the  strong convergence of  $\F^{-1}_{\rg, \delta}$,   and the strong two-scale convergence of   $\mathbb E^\delta$ to $\mathbb E$, together with the   Fr\'echet-Kolmogorov compactness theorem,  yields the last convergence in \eqref{convergence_2}.
\end{proof}

%%%%%%%%%%%%%%%%%%%%%%%%%%%%%%%%%%
\section{Derivation of macroscopic equations for microscopic model~\eqref{growth_micro_1} and~\eqref{micro_model_ref_2}} \label{section_deriv}
%%%%%%%%%%%%%%%%%%%%%%%%%%%%%%%%%%

We shall use both the formal asymptotic expansion and  two-scale convergence methods to derive macroscopic equations for~\eqref{growth_micro_1},\eqref{micro_model_ref_2}. 

\subsection{Formal asymptotic expansion} First we present the formal derivation of the macroscopic equations using the asymptotic expansion of $\bu^\delta$ and $\F_{\rg, \delta}^{-1}$ in powers of $\delta$. The  convergence results  \eqref{convergence_1} ensure  that the limit functions (zero-order terms)  $\bu=\bu(t,x)$  and $\F_{\rg} = \F_{\rg}(t,x)$  are independent of the microscopic variable $y \in Y_w$. 
Hence the formal asymptotic expansion ansatz reads
 \begin{equation}\label{ansatz_1} 
 \begin{aligned}
 \bu^\delta(t,x) & = \bu(t,x) + \delta \bu_1(t,x,x/\delta) + \delta^2 \bu_2(t,x,x/\delta) + \ldots, \\
  \F^{-1}_{\rg, \delta}(t,x) & = \F^{-1}_{\rg}(t,x) + \delta \F^{-1}_{\rg, 1}(t,x,x/\delta) + \delta^2 \F^{-1}_{\rg, 2}(t,x,x/\delta) + \ldots, 
 \end{aligned}
 \end{equation}
 for $(t,x)\in (0,T)\times \Omega^\delta$  and  $\bu_j(t,x, \cdot)$ and $ \F^{-1}_{\rg, j}(t,x, \cdot)$ are $Y$-periodic, for $j=1,2, \ldots$.
 Substituting \eqref{ansatz_1} into microscopic equations~\eqref{micro_model_ref_2}  yields 
 \begin{equation}\label{expan_1}
\begin{aligned} 
 -({\rm div}_x + \frac 1 \delta {\rm div}_y)\Big( J_{\rg}^\delta \Big[
 \mathbb E(x,y)\,  \bsve^{\rm el}\big(   \nabla_x \bu + \nabla_y \bu_1 + \delta(\nabla_x \bu_1 + \nabla_y \bu_2) & \ldots , \F_{\rg}^{-1} + \delta \F_{\rg, 1}^{-1}\ldots \big) \\
 + P_1(t,x){\bf I} \Big]\big( \F_{\rg}^{-T} + \delta \F_{\rg, 1}^{-T} \ldots \big) 
  \Big)
 & =  - J_{\rg}^\delta\, \big(\F_{\rg}^{-T} + \delta \F_{\rg,1}^{-T}  \ldots\big) \nabla_x P_1(t,x)
  \end{aligned} 
\end{equation} 
in $(0,T)\times \Omega \times Y_w$, where $J_{\rg}^\delta = J_{\rg}^\delta \big( (\F_{\rg}^{-1} + \delta \F_{\rg, 1}^{-1}  \ldots)^{-1} \big)$.
% and $\F_{\rg, \delta}  = \F_{\rg, \delta}( \nabla_x \bu_0 + \nabla_y \bu_1 + \delta(\nabla_x\bu_1 + \nabla_y \bu_2) \ldots)$.  
%
For the boundary conditions 
on $(0,T)\times \Omega \times \Gamma$ we have 
 \begin{equation}\label{expan_boun1}
\begin{aligned} 
 & J_{\rg}^\delta \Big[ \bbE(x,y) \, \bsve^{\rm el}\big( \nabla_x \bu + \nabla_y \bu_1 + \delta(\nabla_x \bu_1 + \nabla_y \bu_2) \ldots, \F_{\rg}^{-1} + \delta \F_{\rg, 1}^{-1}\ldots \big)\\
 & \qquad \qquad + P_1(t,x) {\bf I} \Big]   \big(\F_{\rg}^{-T} + \delta \F_{\rg,1}^{-T} \ldots\big) N   =  - \delta J_{\rg}^\delta  P_2(t,x,y) \,  ( \F_{\rg}^{-T} + \delta \F_{\rg, 1}^{-T} \ldots)  N 
 \end{aligned} 
\end{equation}
and on  $(0,T)\times \Gamma_N\times Y_w$ we obtain
 \begin{equation}\label{expan_boun2}
\begin{aligned} 
  J_{\rg}^\delta \Big[\bbE(x,y) \, \bsve^{\rm el}\big(  \nabla_x \bu + \nabla_y \bu_1 + \delta(\nabla_x \bu_1 + \nabla_y \bu_2) \ldots, \F_{\rg}^{-1} + \delta \F_{\rg, 1}^{-1}\ldots \big) + P_1(t,x) {\bf I} \Big]  \big(\F_{\rg}^{-T} + \delta \F_{\rg,1}^{-T} \ldots\big) N
 \\ = J_{\rg}^\delta  \Big[  P_1(t,x)  \big(\F_{\rg}^{-T} + \delta \F_{\rg,1}^{-T} \ldots\big) N  +   \bbf(t,x) | ( \F_{\rg}^{-T} + \delta \F_{\rg, 1}^{-T} \ldots) N|  \Big]  .
 \end{aligned} 
\end{equation} 
For the growth equation it holds 
 \begin{equation}\label{growth_formal_11} 
 \partial_t ( \F_{\rg}^{-1} + \delta \F_{\rg, 1}^{-1}\ldots)^{-1} =  G^\delta(x, \nabla_x \bu + \nabla_y \bu_1 + \delta(\nabla_x \bu_1 + \nabla_y \bu_2)  \ldots, \F_{\rg}^{-1} + \delta \F_{\rg, 1}^{-1} \ldots) ( \F_{\rg}^{-1} + \delta \F_{\rg, 1}^{-1}\ldots)^{-1} . 
\end{equation}  

Now we shall consider terms for different powers of $\delta$.   
For $O(\delta^{-1})$  in equation \eqref{expan_1}  and   $O(1)$  in boundary condition \eqref{expan_boun1}   we obtain 
 \begin{equation} \label{delta-1_orderBVP}
 \begin{aligned} 
- {\rm div}_y\Big[  J_{\rg}
 \Big(\mathbb E(x,  y) \bsve^{\rm el}(  \nabla_x \bu + \nabla_y \bu_1, \F_{\rg}^{-1}) + P_1(t,x){\bf I} \Big)\F_{\rg}^{-T}  \Big] &= 0 \; \; && \text{ in } (0,T)\times \Omega\times Y_w,
\\
 J_{\rg} \Big(\bbE(x, y) \bsve^{\rm el}\big( \nabla_x \bu + \nabla_y \bu_1, \F_{\rg}^{-1}\big)   + P_1(t,x){\bf I} \Big)\F_{\rg}^{-T} N & = 0 &&  \text{ on } (0,T)\times \Omega\times \Gamma, \\
 \bu_1 & &&  \: Y-\text{periodic},
 \end{aligned} 
\end{equation}  
where $J_{\rg} = J_\rg(\F_{\rg})$. 
This is an elliptic  problem for $\bu_1$ in $Y_w$ for  given $\bu$ and $\F_{\rg}$. 

Terms of $O(1)$ in equations   \eqref{expan_1} are 
 \begin{equation}\label{eq_u2}
\begin{aligned} 
 &- {\rm div}_x \Big[  J_{\rg}  \Big(
 \mathbb E(x,y)\, \bsve^{\rm el}(\nabla_x \bu  +  \nabla_y  \bu_1, \F_{\rg}^{-1})  + P_1(t,x) {\bf I} \Big) \F_{\rg}^{-T} 
\Big]  
\\
& - {\rm div}_y \Big[ J_{\rg}^\prime \cdot \F_{\rg,1}^{-1} \, 
 \Big(\mathbb E(x,y)\,  \bsve^{\rm el}(\nabla_x \bu  +  \nabla_y  \bu_1, \F_{\rg}^{-1}) + P_1(t,x) {\bf I} \Big)  \F_{\rg}^{-T} \Big]
 \\
 & - {\rm div}_y \Big[  J_{\rg} 
 \Big(\mathbb E(x,y)\,  \bsve^{\rm el}(\nabla_x \bu  +  \nabla_y  \bu_1, \F_{\rg}^{-1}) + P_1(t,x) {\bf I} \Big) \F_{\rg,1}^{-T}  \Big]\\
 & - {\rm div}_y \Big[  J_{\rg} \, 
 \mathbb E(x,y)\,  \partial_2  \bsve^{\rm el}(\nabla_x \bu +  \nabla_y  \bu_1, \F_{\rg}^{-1}) \F_{\rg,1}^{-1}     \F_{\rg}^{-T} \Big]\\
 &  - {\rm div}_y \Big[  J_{\rg} \, 
 \mathbb E(x,y)\, \partial_1 \bsve^{\rm el}(\nabla_x \bu  +  \nabla_y  \bu_1, \F_{\rg}^{-1})  (\nabla_x \bu_1 + \nabla_y \bu_2)  \F_{\rg}^{-T} \Big]=  - J_{\rg}   \F_\rg^{-T} \nabla_x P_1(t,x) \quad \text{ in } \Omega \times Y_w
  \end{aligned} 
\end{equation} 
and $O(\delta)$-terms in boundary condition \eqref{expan_boun1}   are 
 \begin{equation}\label{bc_u2}
\begin{aligned} 
 J_{\rg}^\prime \cdot \F_{\rg,1}^{-1} \, \Big(\mathbb E(x,y)\, \bsve^{\rm el}( \nabla_x \bu +  \nabla_y \bu_1, \F_{\rg}^{-1})  + P_1(t,x) {\bf I} \Big) \F_{\rg}^{-T}  N \\
+ J_{\rg} \,  \Big(\mathbb E(x,y)\,  \bsve^{\rm el}(\nabla_x \bu +  \nabla_y \bu_1, \F_{\rg}^{-1}) + P_1(t,x) {\bf I} \Big) \F_{\rg,1}^{-T}\, N \\ 
+ J_{\rg} \,  \mathbb E(x,y)\, \partial_1  \bsve^{\rm el}(\nabla_x \bu +  \nabla_y \bu_1, \F_{\rg}^{-1})  (\nabla_x \bu_1 + \nabla_y \bu_2)   \F_{\rg}^{-T}   N \\
+ J_{\rg} \,  \mathbb E(x,y)\,\partial_2  \bsve^{\rm el}(\nabla_x \bu +  \nabla_y \bu_1, \F_{\rg}^{-1}) \F_{\rg, 1}^{-1} \F_{\rg}^{-T}    N  & = - J_{\rg}\,  P_2(t,x, y)  \F_{\rg}^{-T}  N \quad \text{ on } \Omega\times \Gamma,
 \end{aligned} 
\end{equation} 
where $ J_{\rg}^\prime \cdot \F_{\rg,1}^{-1} $ denotes the sum of the components of  the  Hadamard product of two matrices $J_{\rg}^\prime$ and $\F_{\rg,1}^{-1}$.

Equations \eqref{eq_u2}, with boundary condition \eqref{bc_u2} and $Y$-periodicity of $\bu_2$, define an elliptic problem for $\bu_2$ in $Y_w$ for given $\bu$, $\bu_1$, $\F^{-1}_{\rg}$, and $\F^{-1}_{\rg,1}$. Applying the Fredholm alternative to ensure existence of a solution $\bu_2$ of \eqref{eq_u2} and \eqref{bc_u2}, see e.g.~\cite{Evans}, and using $Y$-periodicity of $\mathbb E$, $ \F_{\rg,1}^{-1}$, and  $\bu_j$, for $j=1,2$,   imply macroscopic equations for $\bu$:
 \begin{equation}\label{eq_u0_2}
\begin{aligned} 
 -{\rm div}_x\int_{Y_w} J_{\rg} \, \,\Big(
 \mathbb E(x, y) \bsve^{\rm el}(\nabla_x \bu  +  \nabla_y  \bu_1, \F_{\rg}^{-1}) + P_1(t,x) {\bf I} \Big) \, \F_{\rg}^{-T}
   dy   \\ =  - |Y_w| J_{\rg} \, \F_{\rg}^{-T} \nabla_x P_1(t,x)
   -  \int_\Gamma J_{\rg} \,  P_2(t, x, y)  \F_{\rg}^{-T} N d\gamma_y & \qquad \text{ in } \; (0,T) \times \Omega.
  \end{aligned} 
\end{equation} 
As next we need to determine $\bu_1$.  Using the expression for $\bsve^{\rm el}$
and transformation of  problem  \eqref{delta-1_orderBVP} to be defined on  $Y_{w,\rg} = \F_\rg Y_w$ and 
$\Gamma_\rg =   \F_\rg  \Gamma$ for~$(t,x)\in (0,T)\times \Omega$,  
we  obtain 
\begin{equation} 
 \begin{aligned} 
 {\rm div}_y  
\Big( \mathbb E(x, \F_\rg^{-1} y) \big[  {\rm sym} (\nabla_y \bu_1) + {\rm sym}(\nabla_x \bu  \F_\rg^{-1})  + {\rm sym}( \F_\rg^{-1})  - {\bf I}\big] 
+ P_1(t,x) {\bf I}\Big) &= 0  && \text{ in } Y_{w,\rg}, \\
  \Big( \bbE(x,  \F_\rg^{-1}y)   \big[{\rm sym}(\nabla_y \bu_1)  + {\rm sym}(\nabla_x \bu \F_\rg^{-1} )+ {\rm sym}( \F_\rg^{-1}) -{\bf I}\big]  + P_1(t,x)  {\bf I} \Big) \nu &= 0   && \text{ on }  \Gamma_\rg, \\
  \bu_1 & && Y_\rg - \text{ periodic}.
 \end{aligned} 
\end{equation} 
This is a linear elliptic problem for $\bu_1$ which due to assumptions on $\mathbb E$ has, up to constant in $y$,  a unique solution.  
Thus we can consider $\bu_1$ in the form 
\begin{equation}\label{ansatz_u1}
\bu_1(t, x,y) = \sum_{i,j=1}^d \Big[ (\nabla_x \bu  \F_\rg^{-1})_{ij}(t,x)  +  (\F^{-1}_{\rg}(t,x)- {\bf I})_{ij} \Big] {\bf w}^{ij}(t,x,y) + P_1(t, x) {\bf v}(t,x,y) + \bar \bu_1(t,x),
\end{equation}
where  ${\bf w}^{ij}$ and ${\bf v}$ are solutions of the `unit cell' problems
\begin{equation} \label{unit_cell_11}
 \begin{aligned} 
 {\rm div}_y\big(  
 \mathbb E(x, \F_\rg^{-1} y) \big[ {\rm sym}(\nabla_y {\bf w}^{ij})  +  {\bf b}_{ij}\big] \big) & = 0 
 \; \; && \text{ in }  Y_{w,\rg}, \\
  \mathbb E(x, \F_\rg^{-1} y) \big[ {\rm sym}(\nabla_y {\bf w}^{ij})  +  {\bf b}_{ij} \big] \nu & = 0 
&&  \text{ on } \; \; \Gamma_{\rg}, \quad 
   {\bf w}^{ij}  \quad    Y_{\rg}-\text{ periodic},  
 \end{aligned} 
\end{equation} 
for $i,j=1, \ldots, d$, 
 where ${\bf b}_{ij} = \frac 12  ({\bf e}_i \otimes {\bf e}_j + {\bf e}_j \otimes {\bf e}_i)$ and $\{ {\bf e}_j\}_{j=1}^d$  is the standard basis in $\mathbb R^d$, 
 and 
 \begin{equation} \label{cell_prob_pressure} 
 \begin{aligned} 
 {\rm div}_y\Big( 
 \mathbb  E(x,\F_\rg^{-1} y)\, 
 {\rm sym}(\nabla_y {\bf v}) +  {\bf I} \Big) & = 0 && \text{ in } Y_{w, \rg}, \\
  \big(  \mathbb E(x, \F_\rg^{-1} y)\, 
  {\rm sym}(\nabla_y {\bf v} ) +  {\bf I} \big) \nu & = 0 && \text{ on } \Gamma_{\rg} , \quad 
  {\bf v}  \quad   Y_{\rg}-\text{ periodic}.  
\end{aligned} 
\end{equation} 
Transforming back to the fixed domain  $Y_w$ gives
\begin{equation} \label{unit_cell_12}
 \begin{aligned} 
 {\rm div}_y\Big(  J_\rg
 \mathbb E(x,y) \big({\rm sym}(\nabla_y {\bf w}^{ij} \F_\rg^{-1}) +  {\bf b}_{ij}\big) \F_\rg^{-T} \Big) & = 0 
 \; \; && \text{ in } \; \;  Y_{w}, \\
 J_\rg \mathbb E(x,y) \big( {\rm sym}(\nabla_y {\bf w}^{ij}  \F_\rg^{-1}) +  {\bf b}_{ij} \big) \F_\rg^{-T} N & = 0 \; \; 
 && \text{ on } \; \; \Gamma, \quad {\bf w}^{ij} \; \; \; Y -\text{periodic},
 \end{aligned} 
\end{equation}  
and 
\begin{equation} \label{unit_cell_22}
 \begin{aligned} 
 {\rm div}_y\Big(  J_\rg
 \big(\mathbb E(x,y) {\rm sym}(\nabla_y {\bf v}\F_\rg^{-1}) + {\bf I} \big) \F_\rg^{-T}\Big)& = 0
 \; \; && \text{ in } \; \;  Y_{w}, \\
 J_\rg \big( \mathbb E(x,y)  {\rm sym}(\nabla_y {\bf v}  \F_\rg^{-1}) +  {\bf I} \big) \F_\rg^{-T} N & = 0 \; \; 
 && \text{ on } \; \; \Gamma, \quad {\bf v} \; \; \; Y -\text{periodic}.
 \end{aligned} 
\end{equation}  
Using the solutions of the `unit cells' problems, the macroscopic (homogenized) elasticity tensor $\mathbb E_{\rm hom}$ is defined as
\begin{equation}\label{maacro_E}
\begin{aligned} 
\mathbb E_{{\rm hom}, {ijkl}} (t, x) & =  \frac 1 { |Y_\rg|} \int_{Y_{w, \rg}}  \Big(\mathbb E_{ijkl}(x,\F_\rg^{-1} y) + \Big[\mathbb E(x,\F_\rg^{-1} y)\,  {\rm sym}(\nabla_y{\bf w}^{ij})\Big]_{kl}  \Big) dy
\\
&=  \frac 1 { |Y|}\int_{Y_{w}}  \Big(\mathbb E_{ijkl}(x, y) + \Big[\mathbb E(x,y) \, {\rm sym}(\nabla_y {\bf w}^{ij} \F_\rg^{-1}) \Big]_{kl}  \Big) dy \quad \text{ for } \;  (t,x) \in (0,T)\times \Omega,
\end{aligned} 
\end{equation}
 and matrix $K_{\rm hom}$ is given by 
\begin{equation}\label{maacro_K}
 K_{\rm hom}(t, x) = \frac 1 { |Y_\rg|} \int_{Y_{w, \rg}}
 \mathbb E(x,\F_\rg^{-1} y) \, {\rm{sym}}(\nabla_y {\bf v}) \, dy = \frac 1 { |Y|}\int_{Y_{w}}
 \mathbb E(x,y)\,  {\rm{sym}}(\nabla_y {\bf v} \F_\rg^{-1}) \, dy
\quad \text{ in  }  (0,T)\times \Omega. 
\end{equation}

Using the structure of $\bu_1$ in \eqref{eq_u0_2} and in $O(1)$-terms in \eqref{expan_boun2}, together with the expressions for $\mathbb E_{\rm hom}$ and $K_{\rm hom}$, yields the macroscopic problem  \eqref{macro_main} formulated below.  Considering $O(1)$-terms in \eqref{growth_formal_11} and using again the structure of $\bu_1$ imply macroscopic equations \eqref{macro_growth_eq} for the growth tensor $\F_\rg$.

\subsection{Rigorous derivation of macroscopic model}  Now, using the two-scale convergence  method, see e.g.~\cite{Allaire92, CD, Nguetseng89},  and  the properties of the periodic unfolding operator, see e.g.~\cite{CDG_book},  we rigorously derive the macroscopic model for the microscopic problem~\eqref{growth_micro_1} and~\eqref{micro_model_ref_2}.  

\begin{theorem} 
Under assumptions in Lemma~\ref{lem:conv_stong}, up to a subsequence, solutions  $\{ \bu^\delta\}$ and $\{\F_{g, \delta}\}$  of the microscopic problem  \eqref{growth_micro_1} and~\eqref{micro_model_ref_2} converge, as $\delta \to 0$, to solution $\bu \in L^\infty(0,T; V)$ and $\F_\rg \in W^{1,\infty}(0,T; L^q(\Omega))\cap L^\infty((0,T)\times \Omega)$, for any $q \in (1, \infty)$, of the macroscopic problem
\begin{equation}\label{macro_main}
\begin{aligned}
{\rm div}_x 
\Big(J_\rg \, {\boldsymbol{\sigma}_{\rm hom}}\, \F_{\rg}^{-T}   \Big)
& =    J_\rg   \F_{\rg}^{-T}  \frac 1{ |Y|} \int_{\Gamma}   P_2(t,x, y) N  d\gamma_y     \; && \;\text{ in } \, (0,T)\times \Omega, \\
 J_\rg\,  \boldsymbol{\sigma}_{\rm hom} \, \F_{\rg}^{-T} N  
   & =  \Big(1- \frac{ |Y_w|}{|Y|} \Big) J_\rg P_1(t,x) \F_{\rg}^{-T} N +   J_\rg \,   | \F_{\rg}^{-T} N|\,\bbf(t,x)  \;   && \; \text{ on } \,(0,T)\times \Gamma_N, \\
\bu \cdot N =0, \; \; J_\rg \Pi_\tau\big(\boldsymbol{\sigma}_{\rm hom} \F_{\rg}^{-T}  N\big)& = \Big(1- \frac{ |Y_w|}{|Y|} \Big) J_\rg P_1(t,x) \Pi_\tau(\F_{\rg}^{-T} N), \; \;  \text{ or } \; \;    \bu = 0 && \; \text{ on } \,(0,T)\times \Gamma_D,  
\end{aligned} 
\end{equation} 
and 
\begin{equation}\label{macro_growth_eq}
\begin{aligned}
\partial_t \F_\rg & = G(x,  \nabla \bu, \F_{\rg}^{-1} ) \, \F_\rg  \qquad && \text{ in } (0,T)\times \Omega,  \\
\F_{\rg}(0) & = {\bf I} && \text{ in } \Omega, 
\end{aligned}
\end{equation}
with 
 \begin{eqnarray}
\tilde G(x,  \nabla \bu, \F_\rg^{-1}) &=& 
 \eta_\sigma[(|Y|/|Y_w|) {\boldsymbol{\sigma}}_{\rm hom} (t,x,\nabla \bu, \F_\rg^{-1})- {\boldsymbol\tau}_\sigma ]_+   \label{G_stress} \qquad \text{ or }  \\
 \tilde G(x,  \nabla \bu, \F_\rg^{-1}) &=&
 \eta_\ve[(|Y|/|Y_w|) {\boldsymbol{\ve}}_{\rm hom} (t,x,\nabla \bu, \F_\rg^{-1}) - {\boldsymbol\tau}_\ve ]_+ , \label{G_strain}
\end{eqnarray}
where 
\begin{eqnarray}
{\boldsymbol{\sigma}_{\rm hom}}(t,x, \nabla \bu, \F_{\rg}^{-1})& =& \mathbb E_{\rm hom}(t,x)  \, {\boldsymbol \ve}^{\rm el}(\nabla \bu, \F_\rg^{-1} ) +  K_{\rm hom}(t,x) P_1(t,x)  \label{sigma_hom}
\\
{\boldsymbol{\ve}_{\rm hom}}(t,x,\nabla \bu, \F_\rg^{-1})
 &= & \frac{|Y_w|}{|Y|} {\boldsymbol \ve}^{\rm el}(\nabla \bu, \F_\rg^{-1} )
+ \sum_{i,j=1}^d  {\boldsymbol \ve}^{\rm el}(\nabla \bu, \F_\rg^{-1} )_{ij}  \frac 1 { |Y|}\int_{Y_{w}}  {\rm sym}\big(\nabla_y {\bf w}^{ij} \F_\rg^{-1}\big) dy \nonumber
\\ &&  +  P_1(t,x) \frac 1 { |Y|}\int_{Y_{w}} {\rm sym}( \nabla_y {\bf v}\F_\rg^{-1} ) dy , \label{strain_hom}
\end{eqnarray}
with ${\boldsymbol \ve}^{\rm el} (\nabla \bu, \F_\rg^{-1} )= {\rm{sym}}(\nabla \bu  \, \F_\rg^{-1}) +  {\rm{sym}}(\F_\rg^{-1} ) -   {\bf I}$, and $G$ is defined by $\tilde G$ as in \eqref{def_G_11}.
The 
macroscopic  tensors $\mathbb E_{\rm hom}$ and $K_{\rm hom}$ are defined by
\eqref{maacro_E} and \eqref{maacro_K}, where  
${\bf w}^{ij}$ and ${\bf v}$ are solutions of the `unit cell' problems \eqref{unit_cell_12} and \eqref{unit_cell_22}.
 The space $V$ is defined in the same way as $V_\delta$ with $\Omega^\delta$ replaced by $\Omega$.

 Assuming, additionally, 
\begin{equation} \label{bound_macro}
\|\nabla \bu \|_{L^\infty((0,T)\times \Omega)} \leq C, 
\end{equation}
solution of problem \eqref{macro_main}, \eqref{macro_growth_eq} is unique and the whole sequence of solutions of microscopic model  \eqref{growth_micro_1} and~\eqref{micro_model_ref_2} converges to the solution of the macroscopic problem.  
\end{theorem} 

\begin{proof} 
Considering $\varphi(t,x) = \psi(t,x) + \delta \psi_1(t,x,x/\delta)$, with $\psi \in C^1([0,T]; C^1(\Omega)\cap V)$ and 
$\psi_1 \in C^1_0((0,T)\times \Omega; C^1_{\rm per} (Y))$, 
as a test function in \eqref{var_inequal}, taking the two-scale limit as $\delta \to 0$, and using the two-scale convergence of $\nabla \bu^\delta$ and the strong convergence of $\F_{\rg, \delta}^{-1}$, ensured by the boundedness of $\F_{\rg, \delta}^{-1}$ and the strong convergence of $\F_{\rg, \delta}$, see Lemma~\ref{lem:conv_stong},  we obtain 
\begin{equation} \label{limit_twoscale_1}
 \begin{aligned} 
 &  \int_{\Omega_T}\frac 1{|Y|} \int_{Y_w}  J_{\rg}\,  \Big[\mathbb E(x,y) \big({\rm sym}(\nabla \bu \, \F_{\rg}^{-1} + \nabla_y \bu_1\F_{\rg}^{-1}) + {\rm sym}(\F_{\rg}^{-1}) - {\bf I}\big) \\
& \qquad \qquad   + P_1(t,x) {\bf I}\Big]  \, (\nabla \psi + \nabla_y \psi_1)\, \F_{\rg}^{-1}\, dy dxdt
  = -
  \int_{\Omega_T} \frac 1 { |Y|} \int_{Y_w} J_{\rg}  \F_{\rg}^{-T} \nabla P_1(t,x)  \psi \, dy dx dt
  \\
& \quad   +
  \int_{(\partial \Omega)_T}  J_{\rg}  P_1(t,x) \F_{\rg}^{-T} N  \psi\,  d\gamma dt + \int_{\Gamma_{N, T}}  J_{\rg}  {\bf f}(t,x) |\F_{\rg}^{-T} N|   \psi\,  d\gamma dt
  \\
& \quad     - \int_{\Omega_T} \frac 1{|Y|} \int_{\Gamma} J_{\rg} \, P_2(t,x, y) \F_{\rg}^{-T} N \,\psi \,  d\gamma_y dx dt,
 \end{aligned}
\end{equation}
where $\Omega_T = (0,T)\times \Omega$, $(\partial \Omega)_T = (0,T)\times \partial \Omega$, and $\Gamma_{N,T} = (0,T)\times \Gamma_N$. 
%Here $C_\Gamma^1(\Omega) =\{u\in C^1(\overline\Omega)^d\, : \,  u= 0 \text{ or } u\cdot N = 0 \text{ on } \Gamma_D \}.$
    Considering now $\psi =0$ implies the problem for~$\bu_1$:
\begin{equation} \label{prob_u1_two-scale}
\begin{aligned}
\int_{\Omega_T} \frac 1{|Y|} \int_{Y_w} J_{\rg}\,\Big( \mathbb E(x, y) \big[  {\rm sym} (\nabla_y \bu_1 \,\F_\rg^{-1}) + {\rm{sym}}(\nabla \bu  \,\F_\rg^{-1})  + {\rm{sym}}( \F_\rg^{-1})  - {\bf I}\big] &
\\
+ P_1(t,x) {\bf I}\Big)    \nabla_y \psi_1 \F_{\rg}^{-1} \, & dydx dt  = 0.
\end{aligned}
\end{equation}
From \eqref{prob_u1_two-scale}, considering for $\bu_1$ the ansatz \eqref{ansatz_u1} we obtain the `unit cell' problems \eqref{unit_cell_11} and \eqref{cell_prob_pressure}  and the formulas for the macroscopic coefficient $\mathbb E_{\rm hom}$ and $K_{\rm hom}$ in \eqref{maacro_E} and \eqref{maacro_K}, respectively. 

Consider $\psi_1 = 0$  in \eqref{limit_twoscale_1} yields
\begin{equation}\label{prob_u_two-scale} 
 \begin{aligned} 
  \int_{\Omega_T}  J_{\rg}\, & \Big[ \frac 1{|Y|} \int_{Y_w}\mathbb E(x,y) \big({\rm sym}(\nabla \bu \,\F_{\rg}^{-1} + \nabla_y \bu_1\F_{\rg}^{-1}) + {\rm sym}(\F_{\rg}^{-1}) - {\bf I}\big) dy + \frac{ |Y_w|}{|Y|} P_1(t,x) {\bf I}\Big]   \nabla \psi \F_{\rg}^{-1} dxdt \\
  = & -
  \int_{\Omega_T} \frac {|Y_w|} { |Y|} J_{\rg} \F_{\rg}^{-T}  \nabla P_1(t,x) \psi\, dxdt +
  \int_{(\partial \Omega)_T} J_{\rg}  P_1(t,x)  \F_{\rg}^{-T} N  \psi \, d\gamma dt \\
 &  + \int_{\Gamma_{N,T}} J_{\rg} {\bf f}(t,x) |\F_{\rg}^{-T} N|  \psi \, d\gamma dt
  - \int_{\Omega_T}  J_{\rg} \F_{\rg}^{-T} \frac 1{|Y|} \int_{\Gamma}  P_2(t,x, y) N  \,  d\gamma_y \,\psi\,  dxdt.
 \end{aligned}
\end{equation}
%Using the Green formula in the boundary integral over $\partial \Omega$ yields 
%\begin{equation} 
 %\begin{aligned} 
  %\int_{\Omega} J_{\rg}\, \frac 1{|Y|} \int_{Y_w}  \mathbb E(x,y) \big({\rm{sym}}(\nabla_x \bu_0 \,\F_{\rg}^{-1}) + {\rm{sym}}(\nabla_y \bu_1\F_{\rg}^{-1} ) + {\rm{sym}}(\F_{\rg}^{-1}) - {\bf I}\big)  dy \, \F_{\rg}^{-T}   \nabla \psi  \, dx
 %\\
 %=  \int_{\Omega}  J_{\rg}  \Big( 1 - \frac{ |Y_w|}{|Y|} \Big) \F_{\rg}^{-T}  \big(\nabla P_1(x) \psi + P_1(x) \nabla \psi \big)  d x + \int_{\Gamma_N} J_{\rg} \,  {\bf f}(x) |\F_{\rg}^{-T} N| \, \psi \, d\gamma 
 %\\
 %\quad - \int_{\Omega}  J_{\rg} \F_{\rg}^{-T} \frac 1{|Y|} \int_{\Gamma}  P_2(x, y) N  \,  d\gamma_y \,\psi\,  dx.
 %\end{aligned}
%\end{equation}
%Considering transformation from $Y$ to $Y_{\rg}$ yields
% \begin{equation}\label{eq_u0_2}
%\begin{aligned} 
 %\int_{\Omega} J_\rg \frac 1 { |Y_\rg|} \int_{Y_{w, \rg}} 
 % \mathbb E(\F_\rg^{-1} y) \big[{\rm{sym}}(\nabla_x \bu_0  \F_\rg^{-1}) + {\rm{sym}}(\nabla_y \bu_1)   +  {\rm{sym}}(\F_\rg^{-1}) -  {\bf I} \big]     dy  \, \F_{\rg}^{-T}  \nabla \psi dx \\
 %=  \int_{\Omega}  J_{\rg}  \Big( 1 - \frac{ |Y_w|}{|Y|} \Big) \F_{\rg}^{-T}  \big(\nabla P_1(x) \psi + P_1(x) \nabla \psi \big)  d x 
 %+ \int_{\Gamma_N} J_{\rg} \,  {\bf f}(x) |\F_{\rg}^{-T} N| \, \psi \, d\gamma 
 %\\
 %\quad - \int_{\Omega}  J_{\rg} \F_{\rg}^{-T} \frac 1{|Y|} \int_{\Gamma}  P_2(x, y) N  \,  d\gamma_y \,\psi\,  dx.
%\end{aligned} 
%\end{equation} 
Using the structure of $\bu_1$ and expressions for $\mathbb E_{\rm hom}$ and $K_{\rm hom}$, we obtain 
\begin{equation}\label{macro_weak_222}
\begin{aligned}
  \int_{\Omega_T}  J_\rg \, \Big[ \mathbb E_{\rm hom} (t, x)  \big[{\rm{sym}}(\nabla \bu \,\F_\rg^{-1} ) + {\rm{sym}} \big(\F_\rg^{-1}\big) -  {\bf I} \big]  + 
\Big(K_{\rm hom}(t,x) + \frac{|Y_w|}{|Y|} {\bf I} - {\bf I}\Big) P_1(t,x) \Big]\, \nabla \psi \, \F_{\rg}^{-1}  \, dxdt
\\
 =  \int_{\Omega_T}  J_{\rg}  \Big( 1 - \frac{ |Y_w|}{|Y|} \Big) \F_{\rg}^{-T}  \nabla P_1(t,x) \, \psi \,  dx dt
 + \int_{\Gamma_{N,T}} J_\rg \,  | \F_{\rg}^{-T} N|\,\bbf(t,x)\,    \psi \, d\gamma dt
\\
-   \int_{\Omega_T}  J_\rg  \, \F_{\rg}^{-T}  \frac 1{ |Y|} \int_{\Gamma}   P_2(t,x, y) N d\gamma_y \,  \psi \, dx dt,
\end{aligned}
\end{equation}
which is the macroscopic  problem \eqref{macro_main} for $\bu$ in the weak form. 
Here we used   ${\rm div} (J_{\rg} \F_{\rg}^{-T})  = 0$ in $(0,T)\times\Omega$ to rewrite the boundary term involving $P_1$ as an integral over $\Omega$ of $\nabla P_1$. Using the same calculations, the first integral on the right hand side  in \eqref{macro_weak_222}, combined with the last two terms on the left hand side,  can be rewritten as the integral over $(0,T)\times \partial \Omega$.

To pass to the limit in the equation \eqref{growth_micro_1}, we first determine the weak limit of $\int_{\delta ( [x/\delta]_Y + Y_w)} {\boldsymbol\sigma}(\tilde x, \nabla\bu^\delta, \F^{-1}_{\rg, \delta} )  d \tilde x$ and  $\int_{\delta ( [x/\delta]_Y + Y_w)} {\boldsymbol\ve}^{\rm el}(\nabla\bu^\delta, \F^{-1}_{\rg, \delta} )  d \tilde x$.  Using the two-scale convergence of $\nabla \bu^\delta$ and strong convergence of $\F_{\rm g, \delta}^{-1}$, together with the relation between the two-scale convergence of a sequence and weak convergence of the unfolded sequence, yields
$$
\begin{aligned} 
 \lim\limits_{\delta \to 0} \int_0^T\int_{\Omega}\frac 1 { |Y_w|}\int_{\delta ( [x/\delta]_Y + Y_w)} {\boldsymbol\sigma}(\tilde x, \nabla\bu^\delta, \F^{-1}_{\rg, \delta} )  d \tilde x \, \phi \,dx dt  =
\lim\limits_{\delta \to 0}\int_0^T\int_{\Omega}\frac 1 { |Y_w|} \int_{Y_w}\mathcal T^\delta \big({\boldsymbol\sigma}(x, \nabla\bu^\delta, \F^{-1}_{\rg, \delta} ) \big) dy \,\phi\, dx dt \\
 = 
\int_0^T\int_{\Omega}\frac 1 { |Y_w|} \int_{Y_w} \mathbb E(x,y) \big[{\rm sym}(\nabla \bu \, \F_\rg^{-1}) + {\rm sym} (\nabla_y \bu_1 \F_\rg^{-1})  +  {\rm sym} (\F_\rg^{-1}) -  {\bf I} \big]  dy\, \phi\, dx dt \\
 = 
 \int_0^T\int_{\Omega} \frac{ |Y|}{|Y_w|} {\boldsymbol \sigma}_{\rm hom}(t,x, \nabla \bu, \F_\rg^{-1}) \, \phi \, dx dt,
\end{aligned} 
$$
for  $\phi \in C_0((0,T)\times \Omega)$ and ${\boldsymbol \sigma}_{\rm hom}(t, x, \nabla \bu, \F_\rg^{-1})$ given by \eqref{sigma_hom}. 
Similarly we obtain the weak convergence for the strain 
$$
 \lim\limits_{\delta \to 0} \int_0^T\int_{\Omega}\frac 1 { |Y_w|}\int_{\delta ( [x/\delta]_Y + Y_w)} {\boldsymbol\ve}^{\rm el}(\nabla\bu^\delta, \F^{-1}_{\rg, \delta} )  d \tilde x \,  \phi \, dx dt = \int_0^T\int_{\Omega} \frac{ |Y|}{|Y_w|}  {\boldsymbol\ve}_{\rm hom} (t,x,\nabla \bu, \F_{\rg}^{-1})\, \phi \, dx dt,
 $$
where  the macroscopic strain ${\boldsymbol\ve}_{\rm hom} (t,x,\nabla \bu, \F_{\rg}^{-1})$ is defined by \eqref{strain_hom}.
Then %applying the unfolding operator $\mathcal T^\delta$ to \eqref{growth_micro_1} and
 the continuity of $G$ and the strong convergence of $\F_{\rg, \delta}$, $\int_{Y_w} \mathcal T^\delta({\rm sym}(\nabla \bu^\delta \F_{\rg, \delta}^{-1} )) dy$, and   $\int_{Y_w} \mathcal T^\delta(\mathbb E^\delta(x)  {\rm sym}(\nabla \bu^\delta \F_{\rg, \delta}^{-1} )) dy$   yield macroscopic equation~\eqref{macro_growth_eq}.

Considering a piece-wise constant approximation of $\F_g$ by $\F_{\rg, j}= {\rm const}$ on $\Omega_j$, with $\Omega = {\rm Int}(\cup_{j=1}^n \overline \Omega_j)$, $\Omega_i \cap \Omega_j = \emptyset$, for $i\neq j$,   $J_{\rg, j} ={\rm det} (\F_{\rg, j}) \geq 1$ and eigenvalues $\lambda_k (\F_{\rg, j}) \geq 1$, for $k=1, \ldots, d$ and $j=1, \ldots, n$, and using the Korn inequality for $\bu(t) \in V$,  we can write
\begin{equation} \label{Korn_macro}
\begin{aligned} 
\int_\Omega |{\rm sym}(\nabla \bu \F_{\rg}^{-1})|^2 dx & =
\lim\limits_{n\to \infty} \sum_{j=1}^n \int_{\Omega_j}  |{\rm sym}(\nabla \bu \F_{\rg,j}^{-1})|^2 dx 
= \lim\limits_{n\to \infty} \sum_{j=1}^n \int_{\F_{\rg,j}\Omega_j}  J_{\rg,j}^{-1} |{\rm sym}(\nabla \bu )|^2 dx \\
&\geq C
 \lim\limits_{n\to \infty} \sum_{j=1}^n \int_{\Omega_j}  |{\rm sym}(\nabla \bu )|^2 dx = C \int_\Omega  |{\rm sym}(\nabla \bu )|^2 dx  \geq C \| \bu \|^2_{V}. 
\end{aligned} 
\end{equation}

Then using \eqref{bound_macro}, in a similar way as in the proof of Lemma~\ref{lem:conv_stong}, for two solutions  $\bu_1, \F_{\rg,1}$ and $\bu_2, \F_{\rg,2}$ of~\eqref{macro_main},\eqref{macro_growth_eq}  we  obtain  
\begin{equation}\label{uniq_estim} 
\|{\rm sym}(\nabla \bu_1 \F_{\rg,1}^{-1}) - {\rm sym}(\nabla \bu_2 \F_{\rg,2}^{-1}) \|^2_{L^2(\Omega)} \leq  C\|\F_{\rg,1}- \F_{\rg,2}\|^2_{L^2(\Omega)}, 
\end{equation} 
for a.a.~$t \in (0,T)$. 
Considering then the equation for $\F_{\rg,1} - \F_{\rg,2}$ and using 
\eqref{uniq_estim} yields $\|\F_{\rg,1}(t)- \F_{\rg,2}(t)\|_{L^2(\Omega)}=0$ for a.a.~$t \in (0,T)$, and hence $\F_{\rg,1} = \F_{\rg,2}$ a.e.~in $(0,T)\times \Omega$. 
Using this in \eqref{uniq_estim} implies 
${\rm sym}(\nabla (\bu_1 - \bu_2) \F_{\rg}^{-1}) =0$ a.e.~in $(0,T)\times \Omega$ and hence, due to \eqref{Korn_macro}, also $\bu_1= \bu_2$ a.e.~in $(0,T)\times \Omega$ and  the uniqueness of  solution of~\eqref{macro_main},\eqref{macro_growth_eq}.  Thus we have that the whole sequence of solutions of the microscopic problem converges to the solution of the macroscopic equations. 
\end{proof} 

\begin{remark}
(i) Using  assumptions on the domain $\Omega$ and regularity of $\mathbb E$, $P_1$, $P_2$ and ${\bf f}$,  estimates \eqref{bound_bu} and~\eqref{bound_macro}, assumed to be true in the rigorous derivation of the macroscopic model,  may possibly be shown considering  approaches similar to~\cite{Nirenberg_59, Nirenberg_64, Mazya_2014, Kenig_2013, Shen_2017}.

(ii) In the derivation of the strong convergence of $\F_{\rg, \delta}$
and $\int_{Y_w} \mathcal T^\delta({\rm sym}(\nabla \bu^\delta \F_{\rg, \delta}^{-1} )) dy$ we assumed the uniform in $\delta$ boundedness \eqref{bound_bu} of the $L^2(Y_w)$-norm of $\mathcal T^\delta (\nabla \bu^\delta)$ in $(0,T)\times \Omega$. 

It is also possible to derive macroscopic equations \eqref{limit_twoscale_1} for $\bu$, by  assuming first the  strong convergence of $\F_{\rg, \delta}$ and then by showing the strong two-scale convergence of ${\rm sym}(\nabla \bu^\delta \F_{\rg, \delta}^{-1})$ deduce the strong convergence of $\F_{\rg, \delta}$. 
To show the strong two-scale convergence of ${\rm sym}(\nabla \bu^\delta \F_{\rg, \delta}^{-1})$, we  consider  $\nabla \bu^\delta$ as a test function in \eqref{var_inequal} and take the limit as $\delta \to 0$. Then  using the lower semi-continuity of the norm, together with the positivity of $J_\rg$ and properties  of $\mathbb E$, yields 
$$
\begin{aligned} 
\frac 1 { |Y|}\Big\langle  J_\rg & \,  \mathbb E(x,y) {\rm sym}((\nabla \bu + \nabla_y \bu_1) \F_{\rg}^{-1}),  {\rm sym}((\nabla \bu + \nabla_y \bu_1) \F_{\rg}^{-1})
\Big\rangle_{\Omega_T, Y_w} \\
& \leq  \lim\limits_{\delta \to 0}\Big\langle  J_\rg^\delta \,  \mathbb E^\delta(x) {\rm sym}(\nabla \bu^\delta \F_{\rg, \delta}^{-1}),  {\rm sym}(\nabla \bu^\delta \F_{\rg, \delta}^{-1}) \Big
\rangle_{\Omega^\delta_T}
\\
& =-\frac 1 { |Y|} \Big\langle  J_\rg\, P_1(t, x) {\bf I} ,  (\nabla \bu + \nabla_y \bu_1) \F_{\rg}^{-1} \Big\rangle_{\Omega_T, Y_w} - \frac 1 { |Y|}  \Big \langle J_{\rg} \, P_2(t,x,y)\, \F_{\rg}^{-T}   N,  \bu \Big \rangle_{\Omega_T, \Gamma} \\
& \quad + \Big \langle J_{\rg} \, P_1(t,x)\, \F_{\rg}^{-T}   N,  \bu \Big \rangle_{(\partial \Omega)_T}  + \frac 1{ |Y|}\Big
\langle J_\rg\,  \mathbb E(x,y) \big[\bI -{\rm sym} (\F_{\rg}^{-1})\big],   (\nabla \bu + \nabla_y \bu_1)  \F_{\rg}^{-1} \Big\rangle_{\Omega_T, Y_w} 
\\
& \quad - \frac 1 {|Y|}\Big
\langle J_\rg\, \F_{\rg}^{-T} \nabla P_1(t,x), \bu \Big\rangle_{\Omega_T, Y_w}
+ \Big\langle  J_\rg\, \bbf(x) \, |\F_{\rg}^{-T} N|, \bu \Big\rangle_{\Gamma_{N,T}}
 \\
 & = \frac 1 { |Y|}\Big\langle  J_\rg \,  \mathbb E(x,y) {\rm sym}((\nabla \bu + \nabla_y \bu_1) \F_{\rg}^{-1}),  {\rm sym}((\nabla \bu + \nabla_y \bu_1) \F_{\rg}^{-1})
\Big\rangle_{\Omega_T, Y_w},
\end{aligned} 
$$
where the last equality follows from \eqref{limit_twoscale_1}  by considering  $\psi = \bu$ and $\psi_1 = \bu_1$. 
This, together with the two-scale convergence of $  \mathbb E^\delta(x) {\rm sym}(\nabla \bu^\delta \F_{\rg, \delta}^{-1})$ and the strong convergence of $J_\rg^\delta$, ensured by the strong convergence of $\F_{\rg, \delta}$, implies the corresponding strong two-scale convergence of ${\rm sym}(\nabla \bu^\delta \F_{\rg, \delta}^{-1})$.  Using now this result in equation \eqref{growth_micro_1} yields the  strong convergence of $\F_{\rg, \delta}$.
\end{remark} 
 
\section{Numerical simulations of microscopic and macroscopic two-scale problems} \label{section_numerics}

To numerically solve the microscopic problem \eqref{micro_model_ref_2}  and \eqref{growth_micro_1} we consider a rectangular tissue of hexagonal cells with geometrical parameters -- wall length $l_1$, $l_2$, wall thickness $w$ and angle $\theta$ -- as shown in the unit cell scheme of Figure~\ref{fig1}a-b. The tissue is composed of  $N_x$ cells along the first axis and $N_x/2$ cell layers along the second axis (see Figure \ref{fig2}d). The bottom cell layer is cut horizontally at mid-height, and we assume no normal displacement and no tangential stress at the cut walls. As long as there is no symmetry breaking, this setting amounts to simulating a square tissue made of $N_x(N_x-1)$ cells. The parameters related to the microscopic mechanical properties of the cell wall are its Young modulus $E$ and its Poisson ratio $\nu$. Regarding growth, we denote for the numerical simulations simply by $\eta$ the extensibility $\eta_\sigma$ or $\eta_\varepsilon$, and simply by ${\boldsymbol\tau}$ the yield threshold tensor ${\boldsymbol\tau}_\sigma$ or ${\boldsymbol\tau}_\varepsilon$ in case of stress based and strain based growth respectively. Furthermore, in all our simulations we consider an isotropic threshold tensor ${\boldsymbol\tau}=\tau{\boldsymbol I}$  with parameter~$\tau$.

For the space-dependent turgor pressure $P_1(x)$ a piece-wise constant approximation  $P^\delta$ of $P_1$ is defined as in~\eqref{form_pressure}, 
 with an appropriate function $P_2$, where $P_2(x,\cdot)$   is   $Y$- periodic.
In the case of a linear function $P_1(x)= \alpha x_1 + \beta$ we have $P^\delta(x) = \beta$ for $x \in [0,\delta)$ and $P^\delta(x) = \beta + \alpha \xi_1\delta$ for $x \in \delta(Y+\xi)$ with $\xi\in \Xi^\delta$ and $P_2(x,y)= - \alpha \hat P(y)$, where $\hat P(y)= y_1$ on $Y$ and $Y$-periodically extended to $\mathbb R^d$.  Notice that for constant $P_1$ we have  $P_2 = 0$. 

In order to solve the microscopic problem at each time-step, we consider its variational formulation (\ref{var_inequal}) and  implement it using the Finite Element Method and the open-source finite element software FreeFEM~\cite{freefem}. The vector components of the displacement fields ${\bf u}$ and test function fields $\varphi$ are represented by P$1$ finite elements, and we use the {\tt sparsesolver} solver of FreeFEM. The mesh is built so that a fixed number of triangle edges is imposed per unit length of domain boundary (in most of the simulations we use 25 triangle edges per cell wall length $l_2$).

The numerical simulations of the macroscopic two-scale problem are based on an algorithm that we schematized in Figure~\ref{fig2}b and detailed in Algorithm \ref{algo_coupled}. Namely, at each time-step, we solve the macroscopic problem (Algorithm \ref{algo_macro}) in order to compute tissue level growth and displacement. To compute the required tissue level material properties, we locally solve the corresponding unit cell problems (Algorithm~\ref{algo_UC}). In doing so, we use two meshes (Figure~\ref{fig2}a) --  a fine mesh $\mathfrak{m}$ for the numerical resolution of the macroscopic problem, and a coarse mesh $\mathfrak{M}$ for the calculation of the  homogenized (macroscopic) properties $\mathbb E_{\rm hom}$ and  $K_{\rm hom}$. For instance for the simulations presented in Figures~\ref{fig2} and~\ref{fig3}, we consider a tissue of width $7$ and  use  $10$ edges per unit length for the fine mesh $\mathfrak{m}$, corresponding to $70$ triangle edges along the width of the tissue. In contrast, the coarse mesh $\mathfrak{M}$ for the same tissue has edge length of~$1$, corresponding to~$7$ triangle edges along the width (see Figure~\ref{fig2}a).

To determine  $\mathbb E_{\rm hom}$ and  $K_{\rm hom}$ for given $\F_{\rg}(t, \hat x)$, where $\hat x$ are the nodal points of the coarse mesh,  we compute numerically  solutions of the `unit cell' problems  \eqref{unit_cell_12} and  \eqref{unit_cell_22},  written  in the weak form as
\begin{eqnarray} 
\int_{Y_{w}} J_\rg  \mathbb E(x, y) \Big[ {\rm sym}(\nabla_y {\bf w}^{ij} \F_\rg^{-1}) +  {\bf b}_{ij} \Big]\nabla_y \phi \, \F_\rg^{-1} dy &=&0 
\qquad \text{ for } \; \; \phi \in H^1_{\rm{per}}(Y), \quad \; i,j=1, \ldots, d, \label{UC_weak} \\
  \int_{Y_{w}}J_{\rg} \Big(  
 \mathbb  E(x,y) {\rm sym}(\nabla_y {\bf v}  \F_\rg^{-1}) +   {\bf I} \Big) \nabla_y \phi \, \F_\rg^{-1} dy& =& 0  \qquad \text{ for } \; \; \phi \in H^1_{\rm{per}}(Y). \label{UCw_weak} 
\end{eqnarray}

We call the solutions ${\bf w}^{ij}$ and ${\bf v}$ of the `unit cell' problems the elementary deformations (see Figure~\ref{fig1}b), as they form the building blocks for the effective material properties of the tissue. If we denote an arbitrary displacement field  by $\bf w$, a generalised gradient and its product with the elastic tensor  by
$$
d[{\bf w}](t,x,y) =  {\rm sym}( \nabla_y {\bf w}(t,x,y)\F_\rg^{-1}(t,x) ) \quad
\text{ and }  \quad 
\mathbb Ed[{\bf w}](t,x,y) = \mathbb E(x,y) {\rm sym}( \nabla_y {\bf w}(t,x,y)\F_\rg^{-1}(t,x)),
$$
and  the mean over the  `unit cell' centered on the node with coordinates $x$ by
$$
\overline{d[{\bf w}]}(t,x) = \frac 1 { |Y|}\int_{Y_{w}} d[{\bf w}](t,x,y)dy \quad 
\text{ and } \quad 
\overline{\mathbb Ed[{\bf w}]}(t,x) = \frac 1 { |Y|}\int_{Y_{w}} \mathbb Ed[{\bf w}](t,x,y)dy,
$$
then the effective material properties \eqref{maacro_E} and \eqref{maacro_K} are computed as
\begin{equation}
\begin{aligned}
& \mathbb E_{{\rm hom}, {ijkl}}
=  \overline{ E_{ijkl}} + (\overline{\mathbb Ed[{\bf w}^{ij}]})_{kl}, \qquad K_{\rm hom} = \overline{\mathbb Ed[{\bf v}]}.
 \end{aligned} 
\end{equation}
On the same nodal points we also compute the components of the vector field
\begin{equation}\label{P2} 
    {\bf P_2}(t,x) = \frac 1{ |Y|} \int_{\Gamma}   P_2(t,x,y) N d\gamma_y.
\end{equation}

Then we interpolate the values for $\mathbb E_{\rm hom}$,  $K_{\rm hom}$ and ${\bf P_2}$ to obtain the corresponding tensor and vector fields defined in the whole domain $\Omega$, so that these can be used in the numerical simulations of the macroscopic equation~\eqref{macro_weak_222}. With $\bu$ the solution of the macroscopic problem at time $t$, we compute the macroscopic strain $\boldsymbol{\ve}_{\rm hom}$ and stress  $\boldsymbol{\sigma}_{\rm hom}$ tensor fields,
\begin{equation}
\label{epssigma_hom}
\begin{aligned}
\boldsymbol{\ve}_{\rm hom}(t,x, \nabla \bu, \F_\rg^{-1}) & =
\overline{\bf 1}(t,x) \boldsymbol{\ve}^{\rm el}(\nabla \bu, \F_\rg^{-1} )
+\sum_{i,j=1}^d\overline{d[{\bf w}^{ij}]} (t,x)\boldsymbol{\ve}^{\rm el}(\nabla \bu, \F_\rg^{-1} )_{ij} + \overline{d[{\bf v}]} (t,x)P_1(t,x)  ,
 \\
 {\boldsymbol{\sigma}_{\rm hom}}(t,x, \nabla \bu, \F_\rg^{-1}) &=\mathbb E_{\rm hom}(t,x)  \boldsymbol{\ve}^{\rm el}(\nabla \bu, \F_\rg )  +  K_{\rm hom}(t,x) P_1(t,x),
 \end{aligned} 
\end{equation}
and the corresponding growth rate at this time step according to the strain or stress hypothesis \eqref{G_strain} or \eqref{G_stress}. Finally we use the Euler method to compute the growth tensor $\F_{\rg}$ for the next time-step according to equation~\eqref{macro_growth_eq}.

\begin{algorithm}[H]\label{algo_coupled}
	\SetKwInOut{Input}{Inputs}
	\SetKwInOut{Output}{Output}
	\Input{reference tissue geometry,\\
	    \ $P_1$, $E$, $\nu$, $\eta$, $\tau$ functions defined on the tissue domain\\
	    time step $dt$ and maximal time $t_{max}$}
	\Output{$\F_g(t)$ field for all time points $t$ in $[0,t_{max}]$,\\
	macroscopic displacement field ${\bf u}(t)$ for all time points $t$ in $[0,t_{max}]$.}
	\BlankLine
	Construct fine mesh $\mathfrak{m}$ and coarse mesh $\mathfrak{M}$ of the tissue;\\
	Initialize $t=0$;\\
	Initialize $\F_g= {\bf I}$ on $\mathfrak{M}$;\\
	Evaluate the functions $E$, $\nu$ on each element of the coarse mesh $\mathfrak{M}$\\
	Define FE fields $P_1$, $\eta$ and $\tau$ on the fine mesh $\mathfrak{m}$\\
	\While{$t\leq t_{max}$}
	{
	    \For{all elements of the coarse mesh $\mathfrak{M}$}
	    {
	    Solve the `unit cell' problems  (Algorithm \ref{algo_UC});
		}
		Reconstruct FE fields of the effective properties $\mathbb E_{\rm hom}$, $K_{\rm hom}$ and ${\bf P_2}$ on the coarse mesh $\mathfrak{M}$;\\
		Interpolate and get the FE fields $E_{\rm hom}$, $K_{\rm hom}$ and  ${\bf P_2}$ on the fine mesh $\mathfrak{m}$;\\
		Solve macroscopic problem on $\mathfrak{m}$ (Algorithm \ref{algo_macro});\\
		Compute strain and stress field on $\mathfrak{m}$;\\
	    Compute growth rate depending on strain or stress;\\
	    Compute new growth tensor $\F_g$ on $\mathfrak{m}$;\\
	    Project $\F_g$ on the coarse mesh $\mathfrak{M}$;\\
		Set $t=t+dt$;
	}
	\caption{Coupled simulation}
\end{algorithm}   

\begin{algorithm}[H]\label{algo_UC}
	\SetKwInOut{Input}{Inputs}
	\SetKwInOut{Output}{Output}
	\Input{`unit cell' geometry ($l_1$, $l_2$, $\theta$, $w$),\\
	    cell wall properties ($E$, $\nu$),\\
	    $P_2$, \\
	    growth tensor $\F_g$}
	\Output{effective properties $\mathbb E_{\rm hom}$, $K_{\rm hom}$ and ${\bf P_2}$}
\BlankLine
    Construct `unit cell' mesh;\\
    Solve the `unit cell' problems \eqref{UC_weak} and \eqref{UCw_weak};\\
    Compute the effective properties $\mathbb E_{\rm hom}$ \eqref{maacro_E}, $K_{\rm hom}$ \eqref{maacro_K} and ${\bf P_2}$ \eqref{P2}.
	\caption{`Unit cell' problems}
\end{algorithm}   

\begin{algorithm}[H]\label{algo_macro}
	\SetKwInOut{Input}{Inputs}
	\SetKwInOut{Output}{Output}
	\Input{fine mesh $\mathfrak{m}$ of the tissue,\\
	    pressure field $P_1$ and growth tensor field $\F_g$ on $\mathfrak{m}$\\
	    homogenized property fields $\mathbb E_{\rm hom}$, $K_{\rm hom}$ and ${\bf P_2}$ on $\mathfrak{m}$
	    }
	\Output{macroscopic displacement field ${\bf u}$ on $\mathfrak{m}$}
\BlankLine
	 Solve macroscopic problem \eqref{macro_weak_222}.
	\caption{Macroscopic problem}
\end{algorithm}   

For the `unit cell' problems we implement periodic boundary conditions on $\partial Y$ and the corresponding Neumann boundary conditions on $\Gamma$.
For the macroscopic problem zero  normal displacement   and  no tangential force are imposed on the lower boundary of the tissue, with zero-force conditions on all other boundaries of $\Omega$. The weak formulations of the `unit cell' problems and the macroscopic problem were implemented using FreeFEM. The vector components of the displacement fields are represented by P2 finite elements for the `unit cell' problems, while they are P1 elements for the macroscopic problem. We use the {\tt sparsesolver} solver of FreeFEM to solve both problems.

While the building blocks Algorithms~\ref{algo_UC} and \ref{algo_macro} are  implemented in FreeFEM, the orchestrating Algorithm~\ref{algo_coupled} is implemented in Python. In particular, at each time step we solve a relatively high number of `unit cell' problems that are independent of each other, and we use the {\tt multiprocessing} module of Python to solve them in parallel.

%%%%%%%%%%%%%%%%%%%%%%%%%%%%%%%%%%
\section{Numerical simulation results}
\label{sec:experiments}
%%%%%%%%%%%%%%%%%%%%%%%%%%%%%%%%%%

In this section we first present the results of the `unit cell' problems, which allow us to compute the effective material properties, and in particular we analyse how the microscopic geometrical parameters as well as growth affect the tissue level properties.
Then we validate the multiscale, coupled simulation algorithm on homogeneous isotropic tissues and homogeneous pressure comparing its outcome to the simulation of the microscopic model of the same setup.
We also validate the multiscale approach on more general configurations. Namely, we consider gradient fields over the tissue of all the microscopic parameters one by one and, whenever possible, compare the tissue deformation computed using the coupled simulation to the one obtained by simulating the same setup using the microscopic model.

\subsection{The `unit cell' problems and the macroscopic (homogenized) material properties}
%==============================================================

\begin{figure}
\begin{center}
\includegraphics[width=0.9\columnwidth]{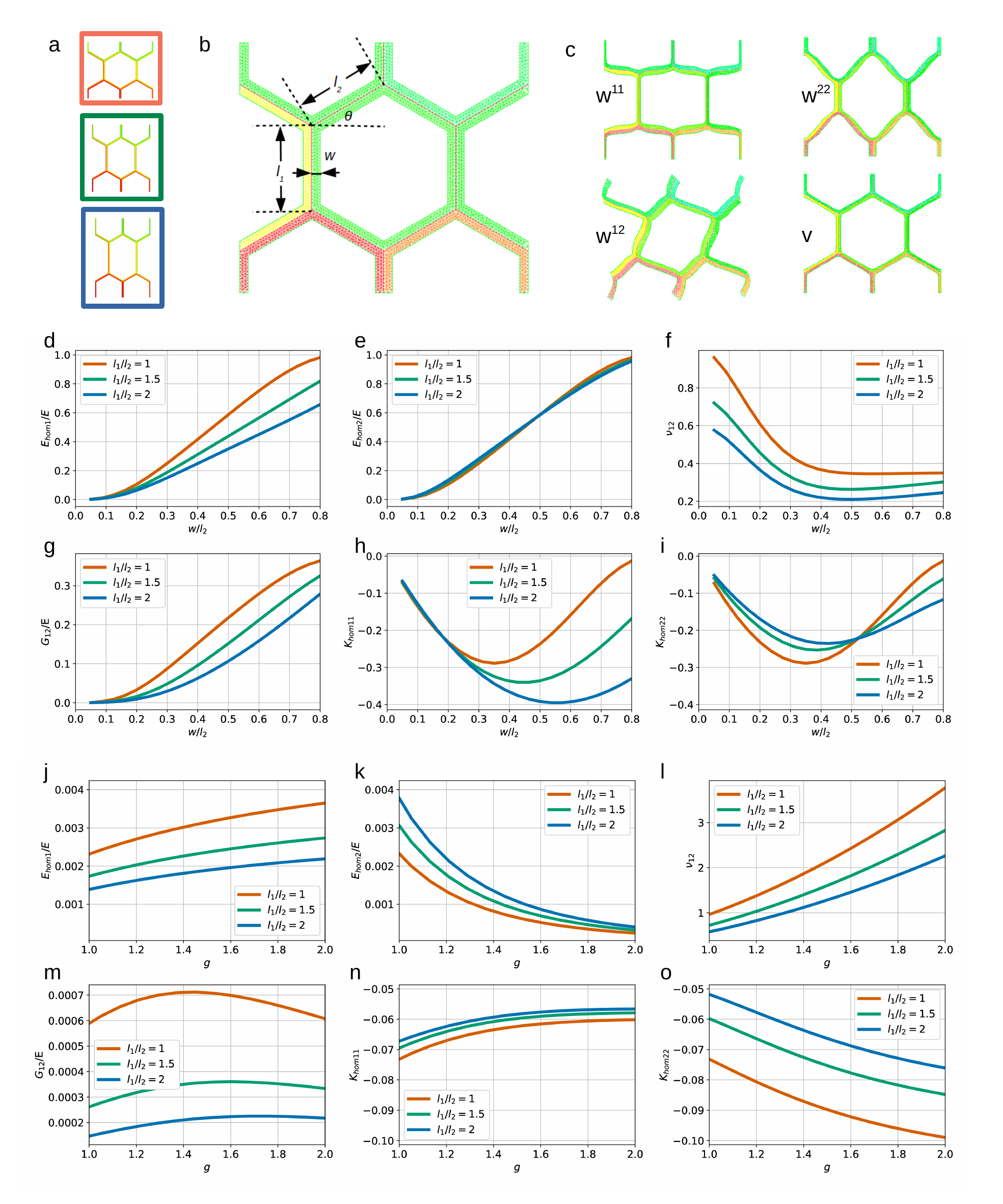}
\end{center}
\caption{\small
{\bf Unit cell problem.} a. `Unit cell' examples: isotropic cell shape $l_1/l_2=1$ (red box), elongated cell shape $l_1/l_2=1.5$ (green box) and $l_1/l_2=2$ (blue box) with $w/l_2=0.05$ and $\theta=30$ for all. The color of the box corresponds to the color of lines in other graphs. b. Geometry of the `unit cell' and definition of the geometrical parameters $l_1$, $l_2$, $w$ and $\theta$;
c. Deformed `unit cell' geometries corresponding to the four elementary solutions of the `unit cell' problems (parameters $E=1$, $\nu=0.35$, $l_1=1$, $l_2=1$, $w=0.1$, $\theta=30$, $\F_g= {\bf I}$);
{\bf d-i. Dependence of homogenized material properties on the relative wall thickness (no growth).} All moduli are normalized by the cell wall Young modulus $E$ and are shown as a function of relative wall thickness $w/l_2$.
d. Modulus along axis 1, $E_{hom1}$; e. Modulus along axis 2, $E_{hom2}$; f. First Poisson's ration $\nu_{12}$; g. Shear modulus, $G_{12}$; 
h-i. Diagonal elements of the effective material property $K_{\rm hom}$. The off-diagonal elements vanish.
{\bf j-o. Dependence of homogenized material properties on the degree of anisotropic growth.}  We consider here only growth along axis 1, with corresponding eigenvalue $g$. All moduli are normalized by the cell wall Young modulus $E$ and are shown as a function of $g$. The relative wall thickness is set  to $w/l_2=0.05$.
j. Modulus along axis 1, $E_{hom1}$; k. Modulus along axis 2, $E_{hom2}$; l. First Poisson's ration $\nu_{12}$; m. Shear modulus, $G_{12}$; 
n-o. Diagonal elements of the effective material property $K_{\rm hom}$. The off-diagonal elements vanish.
}
\label{fig1}
\end{figure}

In the case of non-growing cellular solids, previous studies have computed numerically the homogenized elasticity tensor and compared their results to asymptotic expansions and experimental results, see e.g.~\cite{Malek_2015}. We validated our results by quantitatively comparing the dependence of the effective material properties embedded in $\mathbb E_{\rm hom}$ with the results presented in~\cite{Malek_2015}. Qualitatively, a `unit cell' in the form of a regular hexagon yields an isotropic homogenized medium, see Figure~\ref{fig1}d-e. Elongating the unit cell in the $x_2$-direction only reduces the homogenized modulus in the $x_1$-direction. This is in agreement with the measurements on epidermal onion peels realized using a microextensometer setup \cite{Majda_2022}. The macroscopic (homogenized) Poisson's ratio $\nu_{12}$ is comparatively large for thin walls and is reduced by elongating the unit cell in the $x_2$-direction, see Figure~\ref{fig1}f. The shear modulus turns out to be particularly low for thin walls, and is lower for elongated cells than for regular hexagons, see Figure~\ref{fig1}g. All homogenized properties converge to the microscopic properties when the empty space (cell inside) vanishes, i.e.~$w/l_2\to 1$.

In our multiscale analysis we  introduced a tensorial material property, $K_{\rm hom}$, that accounts for the contribution of cell pressure to the macroscopic (homogenized) stress tensor. For the present choice of coordinate system (which corresponds to symmetry axes of the unit cell), $K_{\rm hom}$ is diagonal; the two diagonal elements of $K_{\rm hom}$ are shown in Figure~\ref{fig1}h-i as  functions of relative thickness. A regular hexagonal unit cell yields an isotropic $K_{\rm hom}$, while an elongated hexagon an anisotropic $K_{\rm hom}$, where the anisotropy increases with wall thickness (as it can be seen from the ratio $K_{\rm hom11}/K_{\rm hom22}$).

In case of isotropic growth, neither $\mathbb E_{\rm hom}$ nor  $K_{\rm hom}$ depends on the degree of the growth. In contrast, when growth is anisotropic, the effective material properties change with growth anisotropy. In order to illustrate this, we show in Figure~\ref{fig1}j-o the macroscopic  (homogenized) material properties as a function of the degree of growth~$g$, when there is only growth along axis 1, for a fixed relative wall thickness. The macroscopic (homogenized) Young's modulus in growth direction increases with $g$, see Figure~\ref{fig1}j, while the perpendicular modulus decreases, see Figure~\ref{fig1}k. Inversely, the absolute value of $K_{\rm hom}$ in growth direction decreases with $g$, see Figure~\ref{fig1}n, while the absolute value of the other element increases, see Figure~\ref{fig1}o. The macroscopic (homogenized) Poisson's ratio $\nu_{12}$ increases with $g$. Finally, the macroscopic (homogenized) shear modulus is non-monotonic.

\subsection{Validation and sensitivity analysis for tissues with homogeneous material properties}\label{section_validation}
%===================================================================
\begin{figure}
\begin{center}
\includegraphics[width=0.88\columnwidth]{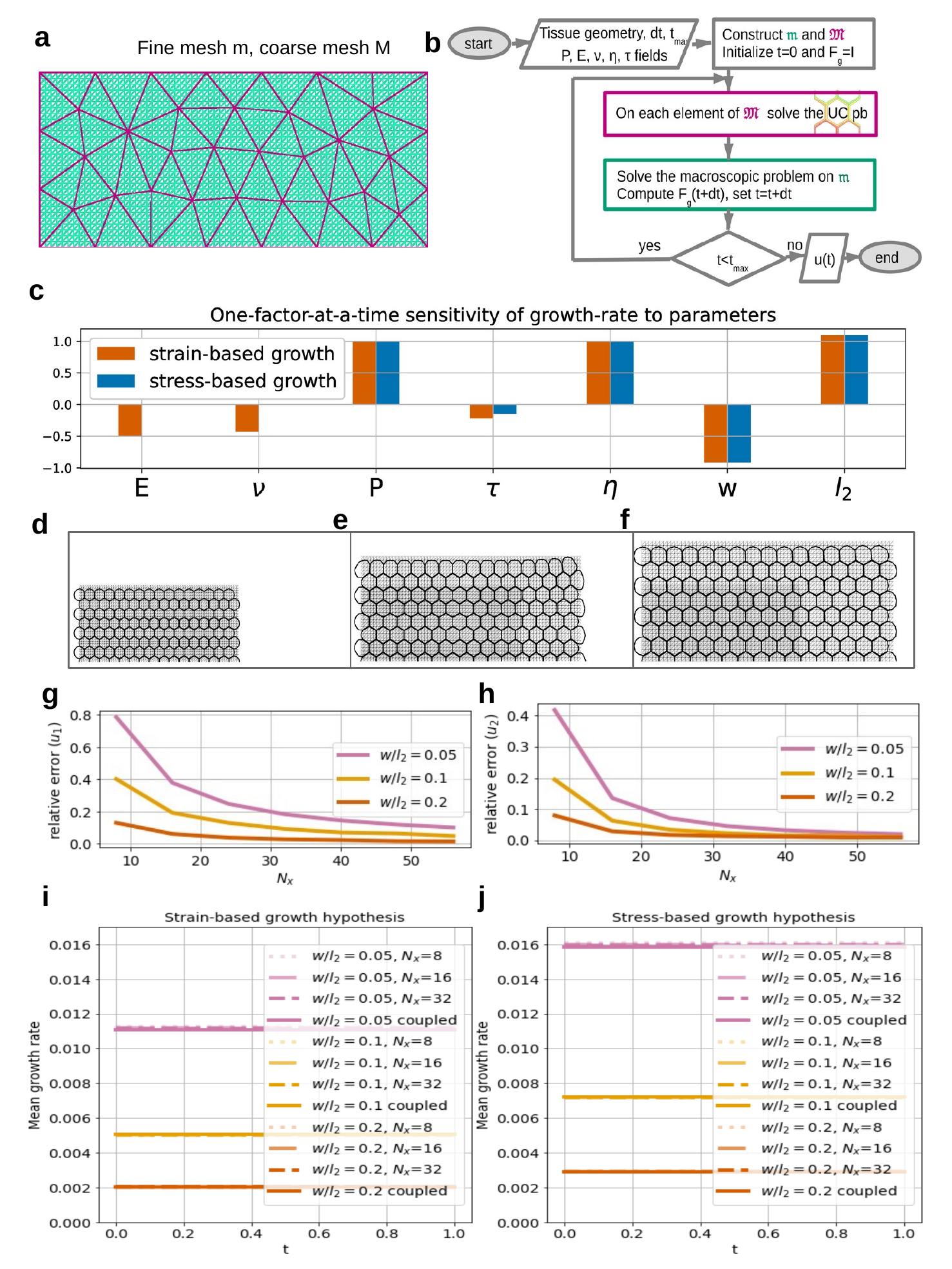}
\end{center}
\caption{\small {\bf The coupled simulation and its validation.} a. The coupled simulation involves in parallel a fine mesh $\mathfrak{m}$ and a coarse mesh $\mathfrak{M}$ of the tissue; b. Simplified diagram of the coupled simulation algorithm, emphasising the usage of the two meshes; c. One-factor-at-a-time sensitivity of growth rate to model parameters around their reference values ($E=1$, $\nu=0.3$, $l_1=l_2=1$, $w=0.05$, $\theta=30$, $\eta=1$, $\tau=0$), for both stress-based and strain-based growth; d-f. Superimposition of the cellular (cell number, $N_x=16$) and continuous representations of tissue at $t=0$  (e) and at $t=60$ using either strain-based (e) or stress-based (f) growth-law -- the smaller shaded area in e-f indicates the initial tissue geometry; g-h. Relative error (difference between coupled simulations and full microscopic model) of the displacement components $u_1$ and $u_2$ as a function of the cell number $N_x$; i-j. Mean growth rate as a function of time, with strain-based (i) or stress-based (j) growth law. The coupled simulation is compared with the microscopic simulation for several values of cell number, $N_x$, and of wall thickness to length ratio, $w/l_2$ -- all but $w/l_2$ parameters are the reference parameters. }
\label{fig2}
\end{figure}

In order to further validate the multiscale method and the macroscopic model derived from the microscopic description of the growth and elastic deformations, we first consider on the one hand a tissue of  $N_x^2/2$, with $N_x=16$, identical regular hexagonal cells with microscopic geometrical parameters the cell wall length $4\delta l_1$, $4\delta l_2$, the cell wall thickness $4\delta w$, and the angle $\theta$, where $\delta = 1/N_x = 1/16$,   $l_1=l_2=1$,  $w=0.05$ and  $\theta=\pi/3$, and on the other hand a continuous tissue of the same size. The cell wall material properties, $E=1$ and $\nu=0.3$, and the pressure, $P_1=0.001$, are homogeneous over the tissue. In Figure~\ref{fig2}d-f, we superimposed the initial states of the microscopic cellular structure and of the continuous tissue, and we visually find that the continuous tissue evolved with the macroscopic  model corresponds well to the cellular tissue evolved with the microscopic model at $t=60$, both for the strain- and the stress-based growth law, see Figure~\ref{fig2}e-f.

We varied the number of cells by scaling their size while keeping tissue dimensions unchanged, and considered three values of relative wall thickness, $w/l_2=0.05$, $0.1$ and $0.2$. We analysed the relative error given by  the relative difference  of solutions obtained from the coupled macroscopic simulations and microscopic simulations, see~\eqref{error} in Appendix. Figures~\ref{fig2}g-h show the relative error  of the displacement components $u_1$ and $u_2$ at the first time step as a function of the number of cells $N_x$. This relative error is higher for thinner walls and it decays with the number of cells.
Accordingly, the values of the growth rate are in agreement between microscopic and coupled macroscopic simulations, see Figures~\ref{fig2}i-j.
The larger error for a tissue of cells with thin walls compared to the error for a tissue
of cells with thick walls for the same number of cells relates to the fact that  the tissue with thick cell walls has smaller voids and is closer in the approximation to the homogeneous tissue than the tissue with thin cell walls.
The mean growth appears constant in time and space,  with thinner walls growing faster. This analysis verifies that the coupled simulation scheme enables the efficient computation of the tissue-scale behavior for a given microscopic cellular geometry and cell wall material properties, and provides a good approximation to the microscopic description of the problem.

In order to better understand the effects of microscopic parameters on the macroscopic behavior we performed a one-factor-at-a-time sensitivity analysis around reference values. We assumed the tissue to be homogeneous and we computed the sensitivity of growth rate with respect to each parameter, as defined by equation~\eqref{sensitivity} in Appendix, see Figure~\ref{fig2}c. As could be expected, growth rate increases with pressure, $P_1$, and extensibility, $\eta$, while it decreases with yield threshold, $\tau$. Growth rate is less sensitive to yield threshold then to pressure and extensibility. Concerning microscopic geometric parameters (cell size, $l_2$, and wall thickness, $w$) growth is promoted by thinner walls if cell size is constant or by larger cells if wall thickness is constant.

Finally, we note a fundamental difference between the two growth models: while the strain-based growth rate decreases with cell wall Young modulus $E$ and Poisson ratio $\nu$, these microscopic material properties have no effect on the stress-based growth.

\subsection{Validation and predictions based on the dynamics of tissues  with heterogeneous  material properties}
%===================================================================

\begin{figure}
\begin{center}
\includegraphics[width=0.9\columnwidth]{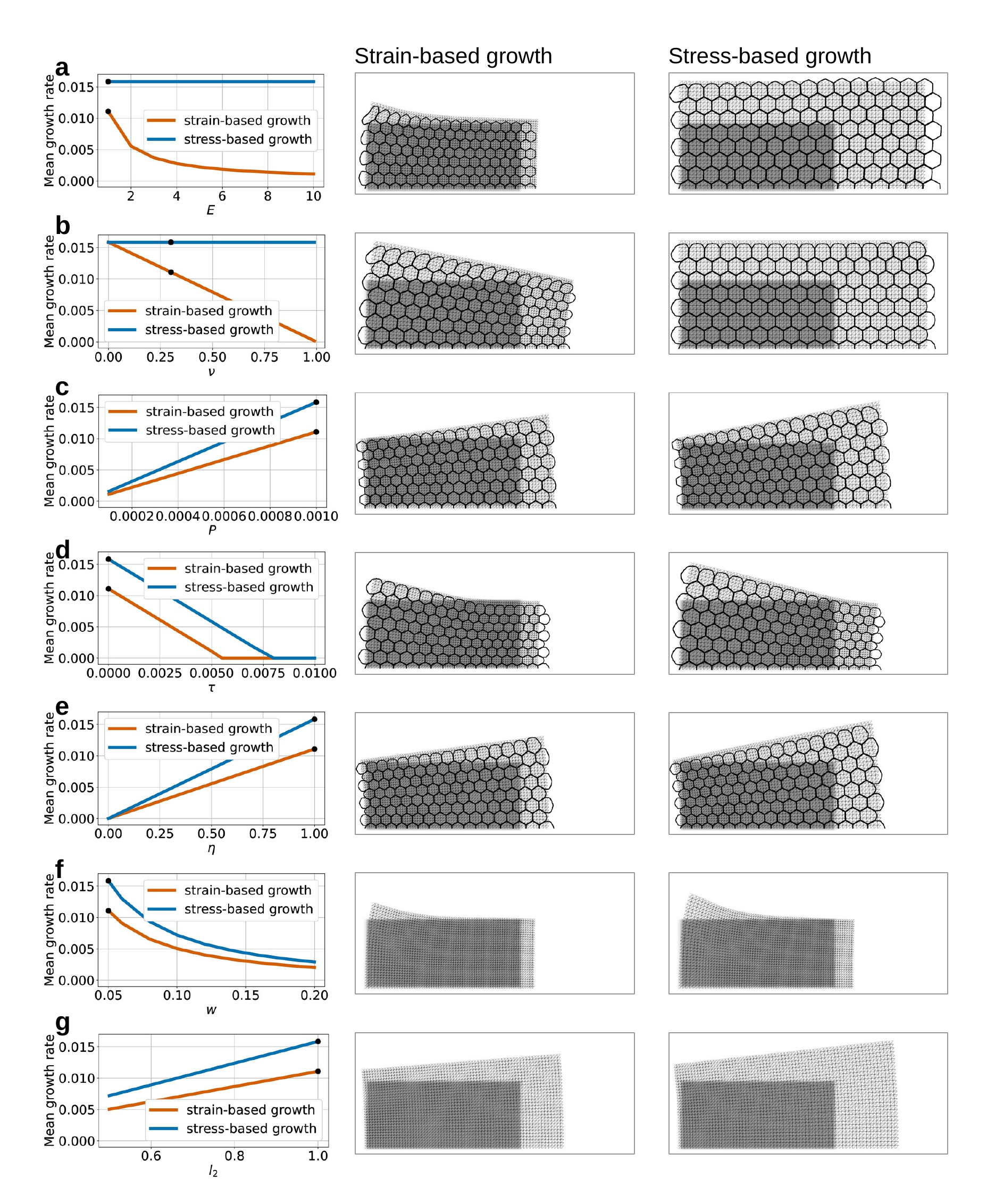}
\end{center}
\caption{\small {\bf Validation and illustration with gradients of material/geometric properties.} First column: One factor-at-a-time parameter scan around the reference model for the homogeneous macroscopic coupled simulation. The results for the reference parameter values are indicated by black dots. Second and third columns: Superimposition of the deformed reference tissues obtained from microscopic (excluding f and g) and from macroscopic coupled simulations, where all parameters are homogeneous, except for the parameter of focus in the row, which varies linearly with first coordinate (in initial geometry) between the minimum and maximum value of the $x$-axis in first column.The shadowed area behind indicates the initial geometry of the tissue. In all simulations $t_{max}=60$, $dt=1$; the number of cells is $N_x=16$ for microscopic simulations. The second and third columns correspond to strain-based and stress-based growth hypotheses, respectively. The parameters varied in each row are: a. Cell-wall Young modulus, $E$; b. Cell-wall Poisson ratio, $\nu$; c. Pressure, $P_1$; d. Strain/stress threshold value, $\tau$; e. Extensibility, $\eta$; f. Wall thickness, $w$; g. Cell size, $l_1=l_2$. }
\label{fig3}
\end{figure}

Now we consider more complex configurations in which material or geometric properties vary spatially, and we use them to further validate the agreement of macroscopic coupled simulations with simulation results for the microscopic model.

Beforehand we perform a parameter scan on larger intervals in the context of homogeneous tissues, and we present in the first column of Figure~\ref{fig3} the growth rate as a function of the parameters for strain- and stress-based growth hypotheses, as predicted by the macroscopic  model. The reference situation is highlighted by black dots on each graph. For several parameters the dependence is affine, so that the sensitivity analysis in Figure~\ref{fig2}c, which was restricted to small variations around the reference values, already captured the essence of the behavior. We find a non-affine dependence of the growth rate on the cell wall Young modulus for strain based growth, and on cell wall thickness for both growth hypotheses. We also note that the results for the strain-based and stress-based growth rates are identical when the Poisson ratio is equal to zero and all other parameters take their reference values, providing an additional check of the numerical simulation codes.

As next,  we consider spatial heterogeneity of each parameter one-by-one, where all other parameters are homogeneous and take the reference value. We consider linear variations of  the parameter along the first axis of the tissue, spanning  the same interval for  which the growth rate variation was presented in the homogeneous context, see the first column of Figure~\ref{fig3}. In the second and third columns of Figure~\ref{fig3} we present the superposition of the corresponding tissues at $t=60$ for the strain-based and stress-based growth hypothesis, respectively. For the  five parameters, $E$, $\nu$, $P_1$, $\tau$, $\eta$, there is a quantitative agreement of tissue shape and size between macroscopic coupled simulations and microscopic simulations. This outcome further validates the multiscale procedure in a context where an additional mechanical stress is induced by spatial differences in growth rate.

Finally, we note that implementing gradients of geometrical properties is not trivial in the microscopic model, whereas such gradients can be easily implemented in the macroscopic coupled simulations. In Figure~\ref{fig3}f-g, we present initial and final tissue shape for gradients in thickness, $w$, and cell size, $l_2$, as predicted by the macroscopic two-scale model.

%%%%%%%%%%%%%%%%%%%%%%%%%%%%%%%%%%
\section{Conclusions}
\label{sec:conclusions}
%%%%%%%%%%%%%%%%%%%%%%%%%%%%%%%%%%
We formulated a microscopic model for the growth of plant tissues and derived the corresponding macroscopic equations, when the limit of the ratio between cell size and tissue size $\delta$ tends to zero. Both methods, the formal asymptotic expansion and rigorous two-scale convergence,  are used to revise the macroscopic equations. Passing to the limit in the  nonlinear equations, resulted from the multiplicative decomposition of the deformation gradient into elastic and growth parts,  requires  the   strong convergence of $\{\F_{\rg, \delta}\}$, the proof of which was the main technical step in the analysis. To show the strong convergence for the sequence of solutions of the microscopic problem and uniqueness  for the macroscopic problem  we used  assumptions  on the boundedness of the deformation gradient, which may be shown using the regularity results for solutions of  the linear elasticity equations.   The multiscale analysis performed here will also apply to other multiscale models for the growth of biological tissues.    

The macroscopic two-scale  model comprises the  coupled system of equations of linear elasticity for the displacement  and  nonlinear ordinary differential equations for the growth tensor,   that are additionally  coupled to the  `unit cell' problems describing cell-scale properties and behaviour. This approach notably enables mapping the parameters of microscopic models that account for cell level characteristics to those of the macroscopic models defined on the tissue level, helping to bridge  two relatively separate worlds in current models of plant growth. We implemented both models (microscopic and macroscopic)  numerically and compared them quantitatively. The macroscopic coupled model requires less computational time for its numerical solution and can be readily adapted to describe spatial variations of parameters. Future work will address more complex spatial patterns of properties, such as accounting for variations in material properties across the cell wall, or additional physics, such as feedback from stress on material properties or hydraulics of water movement during growth. More generally, homogenization and the corresponding multiscale models appear as promising approaches to quantitatively describe morphogenesis.

%\begin{itemize}
%\item relate the presented effective properties to measurable tissue-scale quantities
%\item consider combination of several gradients along the tissue. How their effects superpose?
%\item consider gradients at the microscopic level (across the wall for instance)
%\item introduce hydraulics (water and solute exchange between cells)
%\item consider cell division
%\item consider feed-backs of growth on microscopic material properties 
%\item homogenization at the scale of the cell wall, taking into account its composition and hypothesis on the mechanism of growth
%\end{itemize}

\appendix

\section{Components of the elastic tensor}

The elasticity tensor in two dimensions is implemented as a six-component structure, where every component varies spatially,
$$
\begin{array}{llllllll}
\mathbb E& =&  \{E_0,& E_1,& E_2,& E_3,& E_4,& E_5\} \\
         & =&  \{E_{1111},& E_{2222},& E_{1122},& E_{1212},& E_{1112},& E_{2212}\}\textrm{.}
\end{array}
$$
For an isotropic material in 2 dimensions, or in 3 dimensions with plane stress conditions, the components of the elasticity tensor using Lam\'e's parameters $\lambda$ and $\mu$ are given by
$$
\mathbb E =\{2\mu+\lambda, 2\mu+\lambda, \lambda, \mu, 0, 0\}, 
\quad  \text{ with } \; \;  \mu\ =\ \frac{E}{2(1+\nu)}, \quad 
\lambda\ =\ \frac{E\nu}{1- \nu^2}, 
$$
where $E$  and $\nu$ are  Young's modulus and  Poisson's ratio respectively. 

Given an elasticity tensor $\mathbb E$ in 2 spatial dimensions, for an orthotropic material with symmetry axes along the two Cartesian axes ($E_{1112}=E_{2221}=0$), its material properties, such as  the Young moduli $E_1$ and $E_2$ in the symmetry directions as well as the Poisson ratio $\nu_{12}$ and shear modulus $G_{12}$, can be computed as
$$
 E_1\ =\ E_{1111}-\frac{E_{1122}^2}{E_{2222}},
\quad\quad  E_2\ =\ E_{2222}-\frac{E_{1122}^2}{E_{1111}},
\quad\quad \nu_{12}\ =\ \frac{E_{1122}}{E_{2222}},
\quad\quad  G_{12}\ =\ E_{1212}.
$$

\section{Error and sensitivity definitions}

For  the quantity $a^\delta$ defined in the microscopic domain (tissue) $\Omega^\delta$ and  its corresponding macroscopic quantity $A$ defined in the macroscopic  domain (tissue) $\Omega$, we define  the relative error $e^\delta$  as
\begin{equation}\label{error}
%D^\delta(a^\delta, A)= \frac{a^\delta-A}{\sqrt{\int_{\Omega^\delta} A^2 dx}},\qquad \qquad
e^\delta(a^\delta, A) = \sqrt{\frac{\int_{\Omega^\delta} (a^\delta-A)^2 dx }{\int_{\Omega^\delta} A^2 dx}}.
\end{equation}
The sensitivity $\phi (X,x)$ around the reference model of an output quantity $X$ to the value of parameter $x$ is defined as
\begin{equation}\label{sensitivity}
\phi(X,x) =\frac{x_0}{X(x_0)}\frac{ \partial X}{\partial x}\Big |_{x=x_0},
\end{equation}
where $x_0$ is the reference value of the input parameter $x$ and $X(x_0)$ is the output value in the reference model (all input parameters are equal to their reference value). The sensitivities $\phi$ are normalized, non-dimensional quantities, and thus sensitivities of the output to different input parameters can be compared.

\section{Two-scale convergence and periodic unfolding operator}\label{appendix_A1}
We recall the definition and some properties of the two-scale convergence and periodic unfolding operator.

\begin{definition}[Two-scale convergence]\cite{Allaire92, Nguetseng89}
A sequence $\{u^\delta\}$ in $L^p(\Omega)$, with $1<p<\infty$,  is  two-scale convergent to $u\in L^p(\Omega\times Y)$ 
if for any $\phi\in L^q(\Omega; C_{\rm per}(Y))$, with $1/p+1/q=1$, 
$$
  \lim_{\delta\to 0}\int_\Omega u^\delta(x)\phi\big(x, x/\delta\big) dx
	= \frac 1{|Y|} \int_{\Omega\times Y} u(x,y)\phi(x,y) dydx.
$$
\end{definition}

\begin{theorem}\cite{Allaire96, NeussRadu96}
Let  $\{v^\delta\} \subset L^2(\Gamma^\delta)$ satisfy $\delta \|v^\delta\|^2_{L^2(\Gamma^\delta)} \leq C$,
then there exists a two-scale limit $v\in L^2(\Omega; L^2(\Gamma))$ such that, up to a subsequence,   
$v^\delta$  two-scale converges to $v \in L^2(\Omega\times \Gamma)$ in the sense that
$$
  \lim_{\delta\to 0} \delta \int_{\Gamma^\delta} v^\delta(x)\phi\big(x,x/\delta\big)d\gamma^\delta
	=  \frac 1{|Y|} \int_{\Omega\times \Gamma} v(x,y)\phi(x,y) d\gamma_ydx, \quad \text{ for any }  \phi\in C_0(\Omega; C_{\rm per}(Y)).
$$
\end{theorem}

\begin{lemma}\label{lem1}
(i)\,  If $\{u^\delta \}$ is bounded in $L^2(\Omega)$, there exists
a subsequence (not relabelled) such that $u^\delta \rightharpoonup u$ two-scale as $\delta\to 0$
for some function $u\in L^2(\Omega\times Y)$. 

(ii) \, If $u^\delta\rightharpoonup u$ weakly in $H^1(\Omega)$ then 
$u^\delta \rightharpoonup u$  and $\nabla u^\delta \rightharpoonup \nabla u+\nabla_y u_1$ two-scale, where 
$u_1\in L^2(\Omega;H^1_{\rm per}(Y)/\mathbb R)$.
\end{lemma}

To define the periodic  unfolding operator, let $[z]_Y$ for any $z\in\mathbb R^d$ denote the
unique  combination $\sum_{i=1}^d k_i b_i$ with $k\in\mathbb Z^d$,
such that $z-[z]_Y\in Y$,  see e.g.~\cite{CDDGZ, CDG_book}

\begin{definition}\label{unfold}
Let $p\in[1,\infty]$ and $\phi\in L^p(\Omega)$.  The unfolding 
operator $\T^\delta$ is defined by 
$$
  \T^\delta(\phi)(x,y) = \begin{cases}  \phi\big(\delta \big[x/\delta\big]_Y + \delta y\big)
	\quad & \text{for a.e. }(x,y) \in ( \Omega\setminus \Lambda_\delta ) \times Y, \\
	0 & \text{for a.e. } x\in \Lambda_\delta, \;   y \in Y,
	\end{cases} 
 \quad \text{ and } \;  \T^\delta(\phi)\in L^p(\Omega\times Y).
$$
For $\psi\in L^p(\Gamma^\delta)$ the boundary unfolding operator $\T^\delta_{\Gamma}$ is defined by
$$
  \T^\delta_{\Gamma}(\psi)(x,y) = \begin{cases} \psi\big(\delta \big[x/\delta\big]_Y + \delta y\big)  \quad & \text{for a.e. }(x,y) \in  (\Omega\setminus \Lambda_\delta)  \times \Gamma, \\ 0 & \text{for a.e. } x\in \Lambda_\delta,  \;  y \in \Gamma,
\end{cases}  \quad 
 \text{ and } \; \T^\delta_{\Gamma}(\psi) \in L^p(\Omega\times \Gamma).
$$
For $\psi\in L^p(\Omega^\delta)$  the unfolding operator
$\T^\delta_{Y_w}$ is defined  in \eqref{unfolding_Yw}.
\end{definition}
Notice that in the main text we use the same notation $\mathcal T^\delta$ for all three types of  unfolding operator.

%For any function $\psi$ defined on $\Omega^\delta$,  we have $\T^\delta_{Y_w}(\psi)  = \T^\delta ([\psi]^{\sim})|_{\Omega \times Y_w}$,  with $[\psi]^{\sim}$ denoting extension of $\psi$ by zero  into $\Omega\setminus \Omega^\delta$,  whereas  for $\phi$ defined on $\Omega$, it holds that $\T^\delta_{Y_w}(\phi|_{\Omega^\delta}) = \T^\delta (\phi)|_{\Omega \times Y_w}$. 

%The following result relates two-scale convergence and weak convergence involving the unfolding operator.
\begin{proposition}[\cite{CDG}]\label{prop.unfold}
Let $\{\psi^\delta\}$ be a bounded sequence in $L^p(\Omega)$ for some $1<p<\infty$.   
Then the following assertions are equivalent:\\
(i) $\{\T^\delta (\psi^\delta)\}$ converges weakly to $\psi$ in
 $L^p(\Omega\times Y)$; 
(ii)  $\{\psi^\delta\}$ converges two-scale to $\psi$, $\psi \in L^p(\Omega\times Y)$. 
\end{proposition}

We have the following  properties of the periodic unfolding operator and the boundary unfolding operator: 
\begin{equation}\label{relat_T}
\begin{aligned} 
&  \mathcal T^\delta (F(u,v)) = F( \mathcal T^\delta (u), \mathcal T^\delta (v)), \; \quad  \mathcal T^\delta (v(t, x/\delta)) = v(t,y), \;  x\in \Omega^\delta, y \in Y_w \text{ or } y \in \Gamma, \, t \in (0,T),   \\
& \la \mathcal T^\delta (v),  \mathcal T^\delta (u) \ra_{\Omega_T \times Y_w} =|Y|  \la v, u \ra_{\Omega_{T}^\delta} - |Y|\la v, u \ra_{\Lambda_{\delta,T}} ,\; \; \quad
\|\mathcal T^\delta (\phi)\|_{L^p(\Omega_T\times Y_w)} \leq |Y|^{\frac 1 p } \|\phi\|_{L^p(\Omega_{T}^\delta)},  \\
&\|\T^\delta (\psi)\|_{L^p(\Omega_T\times\Gamma)} \leq \delta^{\frac 1 p } |Y|^{\frac 1 p } \|\psi\|_{L^p(\Gamma^\delta_T)} \leq C \big( \|\psi\|_{L^p(\Omega^\delta_{T})} + \delta \|\nabla \psi\|_{L^p(\Omega^\delta_{T})}\big), 
 \end{aligned} 
 \end{equation}
for   $u, v\in L^2((0,T)\times\Omega^\delta)$ or  $u, v\in L^2((0,T)\times\Gamma^\delta)$, $\phi \in L^p((0,T)\times \Omega^\delta)$,
$\psi \in L^p(0,T; W^{1,p}(\Omega^\delta))$  and $F$ is any linear or nonlinear function, see e.g.\ \cite{CDG, CDDGZ, CDG_book}. 

%We now collect some results on the the convergence of the unfolding of sequences  of functions.
\begin{lemma}[\cite{CDG}]  
(i)  If $\phi \in L^p(\Omega)$, then  $\T^\delta(\phi) \to \phi$ strongly in $L^p(\Omega\times Y)$,  for  $1\leq p < \infty$.

(ii) Let $\{ \psi^\delta\} \subset L^p(\Omega)$,   with $\psi^\delta \to \psi$ strongly in $L^p(\Omega)$, then 
$\T^\delta(\psi^\delta) \to \psi$ strongly in $L^p(\Omega\times Y)$. 

\end{lemma} 

%\begin{theorem} [\cite{CDG, CDDGZ}]\label{conv_unfold}
% Let $\{\psi^\delta \}$ be a bounded sequence in $W^{1, p}(\Omega^\delta)$, for some $1< p < \infty$. Then  there exist functions $\psi \in W^{1, p}(\Omega)$ and $\psi_1 \in L^p(\Omega; W^{1, p}_{\rm per}(Y_w)/\mathbb R)$ such that as $\delta\to 0$,  up to a subsequence,  
%\begin{align*} 
%  & \mathcal T^\delta_{Y_e}(\psi^\delta) \rightharpoonup\psi 
%&& \text{weakly in } L^p(\Omega; W^{1,p}(Y_w)), \\
%  & \mathcal T^\delta_{Y_w}(\psi^\delta) \to  \psi 
% && \text{strongly  in } L^p_{\rm loc}(\Omega; W^{1,p}(Y_w)), \\
 % & \mathcal T^\delta_{Y_w}(\nabla \psi^\delta) \rightharpoonup \nabla \psi  + \nabla_y \psi_1 && \text{weakly in } L^p(\Omega\times Y_w). 
%\end{align*} 
%(ii) \,  Let $\{\phi^\delta\}\subset W^{1, p}(\Omega^\delta)$,  for some $1< p < \infty$,  satisfies
%$$
%\|\phi^\delta\|_{L^p(\Omega^\delta)} + \delta \|\nabla \phi^\delta \|_{L^p(\Omega^\delta)} \leq C.
%$$ 
%Then  there exists   $\phi \in L^p(\Omega; W^{1, p}(Y_w))$ such that as $\delta\to 0$,  up to a subsequence,  
%\begin{align*} 
 % & \mathcal T^\delta_{Y_w}(\phi^\delta) \rightharpoonup\phi 
%&& \text{weakly in } L^p(\Omega\times Y_w), \\
 % &\delta \mathcal T^\delta_{Y_w}(\nabla \phi^\delta) \rightharpoonup  \nabla_y \phi && \text{weakly in } L^p(\Omega\times Y_w). 
%\end{align*} 
%\end{theorem}

\section*{Numerical simulation codes}
Numerical simulation codes for both microscopic model \eqref{growth_micro_1}, \eqref{micro_model_ref_2} and macroscopic model \eqref{macro_main}, \eqref{macro_growth_eq} can be found under \text{https://gitlab.inria.fr/akiss1/planthom}.

\section*{Acknowledgments}
We would like to thank the Newton Institute of Mathematical Sciences (INI) and International Centre for Mathematical Sciences (ICMS) for organising the research programme `Growth, form and self-organisation' and the workshop `Growth, form and self-organisation in living systems'  at which this work was started.
We also acknowledge the support of the Centre Blaise Pascal's IT test platform operated with SIDUS \cite{CBP} at ENS de Lyon for the possibility to test our code on the center's computers.

\bibliographystyle{siamplain}
\bibliography{Master_References}
%\bibliography{references}
\end{document}